\documentclass[10pt]{article}

\usepackage[ngerman,french,english]{babel}

\usepackage[left]{showlabels}                                           

\usepackage{pdfsync}

\usepackage{amsmath}                                                        
\numberwithin{equation}{section}
\usepackage{graphicx}                                                       
\usepackage{subfigure}
\usepackage{enumitem}                                                    
\usepackage{hyperref}
\usepackage{amssymb}                                                        
\usepackage[mathscr]{eucal}                                             
\usepackage{cancel}                                                             
\usepackage[normalem]{ulem}                                                                 
\usepackage{pstricks}
\usepackage{rotating}
\usepackage{lscape}
\usepackage[paperwidth=8.5in,paperheight=11in,top=1.00in, bottom=1.00in, left=1.00in, right=1.00in]{geometry}
\usepackage{mathtools}                                                      
\mathtoolsset{showonlyrefs=true}                                    
\usepackage{fixltx2e,amsmath}                                           
\MakeRobust{\eqref}

\usepackage{dsfont}

\usepackage{mathdots}
\usepackage{amsthm}                                                             
\allowdisplaybreaks

\theoremstyle{plain}
\newtheorem{theorem}{Theorem}
\numberwithin{theorem}{section}

\newtheorem{lemma}[theorem]{Lemma}                              
\newtheorem{proposition}[theorem]{Proposition}
\newtheorem{corollary}[theorem]{Corollary}
\theoremstyle{definition}
\newtheorem{definition}[theorem]{Definition}

\newtheorem{notation}[theorem]{Notation}
\newtheorem{remark}[theorem]{Remark}
\newtheorem{assumption}[theorem]{Assumption}

\def \s {{\sigma}}

\def \a {{\alpha}}
\def \b {{\beta}}

\def \R {\mathbb{R}}
\def \p {\partial}
\def \t {\tau}

\newcommand\tX{\tilde{X}}
\newcommand\tP{\tilde{P}}
\newcommand\tp{\tilde{p}}
\newcommand\ttau{\tilde{\tau}}

\newcommand\Rd{\mathbb{R}^d}
\newcommand\Rdd{\mathbb{R}\times \mathbb{R}^d}
\newcommand\Nzero{\mathbb{N}_0}

\newcommand\Ac{\mathscr{A}}
\newcommand\tAc{\tilde{\mathscr{A}}}
\newcommand\Acc{{\mathscr{A}}}
\newcommand\ac{a}
\newcommand\ta{\tilde{a}}

\newcommand\Lc{\mathscr{L}}

\newcommand\Kc{\mathscr{K}}

\def \caratt {{\mathds{1}}}

\newcommand\eps{\varepsilon}
\newcommand\e{\varepsilon}

\newcommand\Cv{\mathbf{C}}

\renewcommand\d{\delta}

\newcommand\dd{d}

\def \R  {{\mathbb {R}}}
\def \x {{\xi}}

\def \e {{\varepsilon}}
\def \eps {{\varepsilon}}
\def \t {{\tau}}

\def \m {{\mu}}

\def \z {{\zeta}}
\def \p {{\partial}}
\def \a {{\alpha}}
\def \O {{\Omega}}
\newcommand\tO{\tilde{\Omega}}

\def \d {{\delta}}

\def \TT {\mathbf{T}}

\def \caratt {{\mathds{1}}}

\def \a {{\alpha}}
\def \b {{\beta}}
\def \d {{\delta}}

\def \G {\Ga}
\def \tG {\tilde{\Ga}}
\def \tGi {G}

\def \Ga {{\Gamma}}

\def \It\^o {It\^o }
\def \s {{\sigma}}

\def \R {{\mathbb {R}}}
\def \N {{\mathbb {N}}}

\def \x {{\xi}}
\def \e {{\varepsilon}}
\def \eps {{\varepsilon}}

\def \t {{\tau}}
\def \t {{\tau}}

\def \m {{\mu}}

\def \O {{\Omega}}
\def \phi {{\varphi}}

\def \tilde {\widetilde}

\def\l {\lambda}

\def \F {\mathcal{F}}
\def \tF {\tilde{\mathcal{F}}}

\def \Ã  {{\`a }}
\def \è {{\`e }}
\def \ò {{\`o }}
\def \ù {{\`u }}

\def \caratt {{\mathds{1}}}

\newcommand\locdiff{\Ac_t}

\begin{document}

\title{Local densities for a class of degenerate diffusions}

\author{
Alberto Lanconelli
\thanks{Dipartimento di Matematica, Universit\`a di Bari Aldo Moro, Bari, Italy.
\textbf{e-mail}: alberto.lanconelli@uniba.it. Work supported by the Italian {\it INDAM-GNAMPA}}
\and Stefano Pagliarani
\thanks{DIES, Universit\`a di Udine, Udine, Italy.
\textbf{e-mail}: stefano.pagliarani@uniud.it.} \and Andrea Pascucci
\thanks{Dipartimento di Matematica, Universit\`a di Bologna, Bologna, Italy.
\textbf{e-mail}: andrea.pascucci@unibo.it}
}

\date{This version: \today}

\maketitle

\begin{abstract}
We study a class of $\Rd$-valued continuous strong Markov processes that are generated, only
locally, by an ultra-parabolic operator with coefficients that  are regular w.r.t. the intrinsic
geometry induced by the operator itself and not w.r.t. the Euclidean one. The first main result is
a local  It\^o formula for functions that are not twice-differentiable in the classical sense, but only
intrinsically w.r.t. to a set of vector fields, related to the generator, satisfying the
H\"ormander condition. The second main contribution, which builds upon the first one, is an
existence and regularity result for the local transition density. 
\end{abstract}



%
%

\section{Introduction}
\label{sec:intro}
%
%



We study an $\R^d$-valued continuous strong Markov process $X$ that is generated, in a way that
will be specified later, \emph{only locally} on a domain $D\subseteq \R^d$ by the degenerate
operator
\begin{align}\label{ae75bis}
 \Acc_{t}:= \frac{1}{2}\sum_{i,j=1}^{p_0}\ac_{ij}(t,x)\p_{x_{i}x_{j}}+\sum_{i=1}^{p_0}\ac_{i}(t,x)\p_{x_{i}} + \langle B x, \nabla_x \rangle  \qquad t\in [0,T_{0}[,\ x\in D.
\end{align}
Above, $p_0\leq d$ and $B$ is a $(d\times d)$-matrix with constant real entries. In this paper,
the focus is mainly on the case $p_0<d$, which implies that no ellipticity condition on $\Acc_{t}$
is satisfied (i.e. the second order part is fully degenerate). The main structural assumption on
the local-generator $\Acc_{t}$ is the following
\begin{assumption}\label{ass:hypo}
The matrix $B$ is such that the Kolmogorov operator
\begin{equation}\label{eq:kolm_const}
 \Kc:= \frac{1}{2}\sum_{i=1}^{p_0}\p^2_{x_{i}} + \underbrace{\langle B x, \nabla_x \rangle + \partial_t}_{=:Y},\qquad   x\in \mathbb{R}^d,
\end{equation}
is hypoelliptic on $\mathbb{R}\times\mathbb{R}^d$. Equivalently, the vector fields
$\partial_{x_1},\dots,\partial_{x_{p_{0}}}$ and $Y$ satisfy the H\"ormander condition
\begin{equation}
\text{rank } \text{Lie}(\partial_{x_1},\dots,\partial_{x_{p_{0}}},Y)  = d.
\end{equation}
\end{assumption}
Assumption \ref{ass:hypo} is the only hypothesis required for the first main result of the paper, namely the \emph{instrinsic} It\^o's formula. The second main result, about the local density of $X$, is stated under the following additional
\begin{assumption}\label{assum1and}
There exist $N\in\N_0$, $\alpha\in (0,1]$ and $M>0$ 
 such that:
\begin{itemize}
\item[i)] $a_{ij}  ,a_{i}  \in C^{N,\alpha}_B(]0,T_0[\times D)$ for any $i,j=1,\dots,p_0$, with all the (Lie) derivatives bounded by $M$;
  \item[ii)] the following coercivity 
  condition holds on 
  $D$: 
\begin{align}\label{cond:ell-loc}
 M^{-1} |\x|^2 \leq \sum_{i,j=1}^{p_0}
 a_{ij}(t,x)\x_{i}\x_{j}\leq M  |\x|^2 ,\qquad t\in\left]0,T_{0}\right[,\ x\in 
 D,\
 \x\in\mathbb{R}^{p_0}.
\end{align}
\end{itemize}
\end{assumption}
The spaces $C^{N,\alpha}_B(]0,T_0[\times D)$ appearing above are 
the intrinsic H\"older spaces 
induced by the vector fields $\partial_{x_1},\dots,\partial_{x_{p_{0}}}$ and $Y$: their definition
is recalled, for the reader's convenience, in Section \ref{sec:kolmogorov_operators}.

A relevant prototype example that fits Assumptions \ref{ass:hypo} and \ref{assum1and} is the stochastic process $X=(X^1,X^2)$ defined by
\begin{equation}\label{eq:Asian_BS}
\begin{cases}
  \dd X^1_t =X^1_t \dd W_t, \\
  \dd X^2_t = X^1_t \dd t,
\end{cases}
\end{equation}
which is generated by the operator
\begin{equation}\label{eq:bs_asian_kolm}
\Ac = \frac{x^2_1}{2}\partial_{x_1x_1}+x_1\partial_{x_2},\quad\quad (x_1,x_2)\in\mathbb{R}_{>0}^2.
\end{equation}
This operator arises in mathematical finance and is related to the valuation of a class of
path-dependent financial derivatives known as arithmetic Asian options. The process $X^1_t$ is a
geometric Brownian motion and represents the price of a risky asset, whereas $X^2_t$ represents
its average. The operator fulfills Assumption \ref{ass:hypo} in that the commutator
 $$ [\partial_{x_1} , Y] := \partial_{x_1}Y - Y\partial_{x_1}=\partial_{x_2}, $$
and also satisfies Assumption \ref{assum1and} for any $D=\,]a,\infty[$ with $a>0$. Although more
sophisticated models, with more flexible  dynamics (local-stochastic volatility) for the price of
the underlying asset, were proposed to price Asian options, the prototype process
\eqref{eq:Asian_BS} is complex enough to exhibit some interesting mathematical properties. In
fact, the problem of analytically characterizing its joint transition density is still partially open, {and sharp upper/lower bounds were established only recently in \cite{CibelliPolidoroRossi}}. It is easy to recognize in \eqref{eq:bs_asian_kolm} the {double degeneracy} 
of the generator $\Acc$ that our framework allows for: 
on the one hand, the second order part of $\Acc$ is fully degenerate in that 
the partial derivative $\p_{x_{2}x_{2}}$ is missing; 
on the other hand, the coefficient $x_1^2$ of the second order derivative $\partial_{x_1 x_1}$
also degenerates near zero.
More generally (see Proposition \ref{la6} below for the precise statement), the class of stochastic processes that we consider includes locally-integrated diffusions of the form
\begin{align}\label{eq:sde_bis}
 d{X}_{t}=\m(t,{X}_{t})dt+\s(t,{X}_{t})dW_{t},
 \end{align}
with $\mu:[0,T_0[\times\R^d\longrightarrow\R^{d}$ and
$\sigma:[0,T_0[\times\R^d\longrightarrow\R^{d\times n}$ such that, for any $(t,x)\in [0,T_0[\times
D$, 
\begin{align}\label{eq:ste105}
 \mu(t,x)&=\big(a_{1}(t,x), \cdots, a_{p_0}(t,x), 0, \cdots, 0\big) + {B x} ,
 \\
 \s\,\s^\top& =\begin{pmatrix}
    A & 0_{p_0\times (d-p_0)} \\ \label{eq:ste103}
    0_{(d-p_0)\times p_0} & 0_{(d-p_0)\times (d-p_0)} \
  \end{pmatrix}, \qquad A=\left(a_{ij}(t,x)\right)_{i,j=1,\cdots,p_0},
\end{align}
and with $B$ and $a_{ij},a_{i}$ satisfying Assumptions \ref{ass:hypo} and \ref{assum1and}
respectively.

We emphasize that no assumption is required on the generator of $X$ outside the domain $D$,
although the process $X$ ``lives" on $\mathbb{R}^d$, meaning that its trajectories are allowed to
go in and out $D$.

\subsection{Main results and comparison with the literature}
Here we report and discuss the main results of the paper comparing them to the related
literature. 
Granted that precise definitions will be
given in the sequel, namely in Sections \ref{sec:kolmogorov_operators} and \ref{sec:model}, we
will provide here a heuristic 
explanation of all the objects that appear in the
statements below. 
%

The first main result of this paper 
is a local {intrinsic} It\^o formula for $X$. In the following statement, $P_{t,x}$ represents the
probability under which the process $X$ starts from the point $x$ at time $t$ with probability one
and $\F^{t}$ is a filtration to which $(X_T)_{T\geq t}$ is adapted. Moreover, we denote by $\Lc$
the differential operator differential operator
\begin{equation}\label{eq:operator_L}
 \Lc := 
 \frac{1}{2}\sum_{i,j=1}^{p_0}\ac_{ij}(t,x)\p_{x_{i}x_{j}}+\sum_{i=1}^{p_0}\ac_{i}(t,x)\p_{x_{i}} + Y
\end{equation}
where $Y$ is the vector field as defined in \eqref{eq:kolm_const}. 
\begin{theorem}[\bf Intrinsic It\^o formula]\label{cor:ste101}
Let $X$ be a 
\emph{local diffusion on $\R^d$ generated by $\Ac_t$ on $D$} (in the sense of Definition
\ref{def:local_diffusion}) and let Assumption \ref{ass:hypo} be in force. Then, for any fixed
$(t,x)\in\, ]0,T_{0}[\times \R^d$, $\a\in\, ]0,1]$ and
$f\in C^{2,\a}_{B}$ with compact support in $]0,T_{0}[\times D$, 
we have
\begin{equation}\label{ae41}
  f(T,X_{T}) = {f(t,X_{t})}+\int_{t}^{T} 
  \Lc f(u,X_{u})du + M^{t}_{T},\qquad
  {t\leq T<T_0},
\end{equation}
where 
$M^{t}$ is a zero-mean
$\F^{t}$-martingale under $P_{t,x}$, and 
\begin{equation}\label{eq:quad_var}
E_{t,x}\big[|M^t_T|^2\big]=E_{t,x}\bigg[ \int_t^T  \sum_{i,j=1}^{p_0}    a_{ij}(s,X_s)\partial_{x_i} f(s,X_s)\partial_{x_j} f(s,X_s) \dd s\bigg] .
\end{equation}
\end{theorem}
Formula \eqref{ae41} is a local result since no assumption is made on the generator of $X$ outside $D$.
Moreover, the It\^o formula above is stronger than the classical one as it is proved for a class of functions that are not twice-differentiable in the classical sense, 
but only with respect to the non-Euclidean geometry induced by 
the vector fields $\partial_{x_1},\dots,\p_{x_{p_{0}}}$ and $Y$ in Assumption \ref{ass:hypo}. 
Roughly speaking, we say that $f\in C^{2,\a}_{B}\left(]0,T_{0}[\times D\right)$ 
if $\partial_{x_1}f,\dots,\p_{x_{p_{0}}}f$ and $Yf$
exist on $]0,T_{0}[\times D$ 
and 
they are $\a$-H\"older continuous with respect to the semi-distance
 $$|T-t|^{1/2}+\big|y-e^{(T-t)B}x\big|_B, \qquad (t,x),(T,y) \in\mathbb{R}\times \mathbb{R}^d, $$
with $|\cdot|_B$ as in \eqref{e7}.  Note that $Yf$ is meant as a Lie derivative and not as a
combination of Euclidean derivatives: in principle, $\partial_t f$ and $\partial_{x_i}f$ with
$i>p_0$ do not exist; on the other hand, the Euclidean space $C^{2,\alpha}$ is included in
$C_{B}^{2,\alpha}$. We also highlight the fact that Assumption \ref{assum1and} is not required in
Theorem \ref{cor:ste101}.

Just like the classical It\^o formula is based on the standard Taylor expansion, 
the cornerstone of \eqref{ae41} is a non-Euclidean Taylor formula, proved in
\cite{PPP1} and \cite{PagPig} for functions in $C^{2,\alpha}_{B}$, that 
roughly states that
\begin{equation}\label{eq:int_tay_2nd}
f(T,y) = \mathbf{T}^{(2)}_{(t,x)}(T,y) +
\text{O}\big(|T-t|^{\frac{2+\alpha}{2}}\big)+\text{O}\big(\big|y-e^{(T-t)B}x\big|_B^{2+\alpha}\big)
\quad \text{as }  (T,y)\to (t,x)\in\, ]0,T_{0}[\times D,
\end{equation}
where
\begin{align}\label{eq:int_tay_2nd_2}
\mathbf{T}^{(2)}_{(t,x)}(T,y) &= f(t,x) + (T-t) Y
f(t,x)+\sum_{i=1}^{p_0}\big(y-e^{(T-t)B}x\big)_{i}\p_{x_{i}}f(t,x)\\
  &\quad +\frac{1}{2}\sum_{i,j=1}^{p_0}\big(y-e^{(T-t)B}x\big)_{i}\big(y-e^{(T-t)B}x\big)_{j}\p_{x_{i},x_j}f(t,x).
\end{align}
With \eqref{eq:int_tay_2nd}-\eqref{eq:int_tay_2nd_2} at hand, it is possible to outline the main arguments the proof of Theorem \ref{cor:ste101} is built upon. Analogously to the classical case, the key step is proving that 
 \begin{equation}\label{eq:diff_semigroup}
\frac{E_{t,x}[f(T,X_T)]-f(t,x)}{T-t}
  \longrightarrow \Lc f(t,x) \qquad \text{as }T-t\to 0^+,
\end{equation}
uniformly w.r.t. $x\in\mathbb{R}^d$, for any $f\in C_{B}^{2,\alpha}$  with compact support in
$]0,T_{0}[\times D$. Applying \eqref{eq:int_tay_2nd} yields $$
\frac{E_{t,x}[f(T,X_T)]-f(t,x)}{T-t} = \underbrace{\frac{E_{t,x}\Big[
\mathbf{T}^{(2)}_{(t,x)}(T,X_T) -f(t,x) \Big]}{T-t}}_{=:g_1(t,T)}
+\underbrace{\frac{E_{t,x}\Big[\text{O}\big(\big|X_T-e^{(T-t)B}x\big|_B^{2+\alpha}\big)\Big]}{T-t}}_{=:g_2(t,T)}+\,
\text{O}\big(|T-t|^{\frac{\alpha}{2}}\big). $$ It is then clear that \eqref{eq:diff_semigroup}
holds true if
\begin{equation}
g_1(t,T) \longrightarrow \Lc f(t,x),\qquad g_2(t,T) \longrightarrow 0,\qquad \text{as } T-t\to 0^+,
\end{equation}
uniformly w.r.t. $x\in\mathbb{R}^d$. Now, while the proof of the first limit above is quite
straightforward and stems simply from the fact that $X$ is locally generated by $\Ac_t$ on $D$
(see Definition \ref{def:local_diffusion}), the second limit is a deeper result and represents the
main element of novelty in the proof of Theorem \ref{cor:ste101}. In particular, we note that
$g_2(t,T) \longrightarrow 0$ is a consequence of the fact that
\begin{equation}
\lim
_{T-t\rightarrow 0^+}
\frac{E_{t,x}\Big[\caratt_{
\{|X_T -x|<\d\} }\big|X_T-e^{(T-t)B}x\big|_B^{2+\a}\Big]}{T-t}  
=0 ,\qquad
\d>0,
  \end{equation}
uniformly w.r.t. $x\in H$ compact subset of $D$, and we emphasize that the latter is stronger than
the classical general estimate for diffusion processes (see \cite{FriedmanSDE1} or
\cite{MR3726894})
\begin{equation}
\lim
_{T-t\rightarrow 0^+}
\frac{E_{t,x}\big[\caratt_{
\{|X_T -x|<\d\} }|X_T-x|^{2+\a}\big]}{T-t}  
=0 ,\qquad 
\d>0,
  \end{equation}
since the intrinsic quasi-norm $|\cdot|_B$ on $\mathbb{R}^d$ is such that $|x|=\text{o}(|x|_B)$ as
$x\to 0$.

The second main result of the paper is the theorem below that states the existence of a local (on
$D$) transition density $\G(t,x;T,\xi)$  for $X$, reveals its {intrinsic} regularity w.r.t. both
the forward and backward variables and shows that it solves a forward and a backward Kolmogorov
equation on $]t,T_0[\times D$ and $]0,T[\times D$, respectively. Before stating the result, we
need to introduce the last additional assumption, which is only needed to prove the regularity
w.r.t. the backward variables.
\begin{assumption}\label{assum:feller}
$X$ is a \emph{Feller process on $D$}, i.e. for any $T\in\, ]0,T_{0}[$ and bounded {$\phi\in
C(\R^d)$}
the function $(t,x)\mapsto 
E_{t,x}[\phi(X_T)]$ is continuous on $]0,T[\times D$.
\end{assumption}
Note that, since the coercivity condition in Assumption \ref{assum1and}-ii) 
only holds on $D$, the Feller property for the semigroup $
\phi \mapsto E_{t,\cdot}[\phi(X_T)]$ is not ensured. This is due to the fact that the trajectories
of $X$ are allowed to leave and re-enter the domain $D$, but no assumption is made on the
generator of $X$ outside $D$. Had Assumption \ref{assum1and} been satisfied for $D=\mathbb{R}^d$,
the Feller property would stem from PDEs arguments, namely the existence and regularity results
for the fundamental solution of $\Lc$ on $\mathbb{R}^d$ that were proved in \cite{MR1386366} and
\cite{DiFrancescoPascucci2} be means of the so-called parametrix method.
\begin{theorem}
\label{th:main}
Let $X$ be a 
\emph{local diffusion on $\R^d$ generated by $\Ac_t$ on $D$} (in the sense of Definition
\ref{def:local_diffusion}) and let Assumptions \ref{ass:hypo} and \ref{assum1and} be in force.
Then:
\begin{enumerate}
\item[a)] $X$ has a local transition density $\Gamma$ on $D$, namely a non-negative measurable 
function $\G(t,x;T,y)$ defined for any $0< t<T<T_0$ and $x\in\R^d$, $y \in D$, such that
\begin{equation}
p(t,x;T,A) = \int_{A} \G(t,x;T,y) \dd y, \qquad A\in\mathcal{B}(D).
\end{equation}
Furthermore, $\G(t,x;T,\cdot)$ is continuous on $D$ and locally bounded uniformly w.r.t.
$x\in\R^d$;
\item[b)] if $N\geq 2$, then for any $(t,x)\in\, ]0,T_0[\times \R^d$ the function $\Gamma(t,x;\cdot,\cdot)\in C^{N,\alpha}_B(]t,T_0[\times D)$ and solves the forward Kolmogorov equation
\begin{equation}\label{eq:forward_kolm}
\Lc^* u = 0, \qquad \text{on }]t,T_0[\times D,
\end{equation}
where $
\Lc^* $ is the formal adjoint of $
\Lc$;
\item[c)] if Assumption \ref{assum:feller} is also in force, 
then for any $(T,y)\in\, ]0,T_0[\times D$ the function $\Gamma(\cdot,\cdot;T,y)\in
C^{N+2,\alpha}_B(]0,T[\times D)$ and solves the backward Kolmogorov equation
\begin{equation}\label{eq:backward_kolm}
\Lc u = 0, \qquad \text{on }]0,T[\times D.
\end{equation}
\end{enumerate}
\end{theorem}

This statement partially generalizes \cite{kusuoka-stroock}, Sec. 4, where analogous results were
obtained under the assumption that the coefficients of the generator are smooth 
on $D$. In particular, the main assumption in the latter reference is a sort of
local hypoellipticity condition for the generator, expressed in terms of Malliavin's matrix, on a
given domain $D$ of $\mathbb{R}^d$, which reduces to our Assumptions \ref{ass:hypo} and
\ref{assum1and} with $N=\infty$ when $\Ac_t$ is of the form \eqref{ae75bis}. So, if on the one
hand the framework in \cite{kusuoka-stroock} is more general, on the other hand our results are stronger in this
particular setting in that they are carried out by assuming the coefficients belonging to the
intrinsic H\"older space $C^{N,\alpha}_B(]0,T[\times D)$. Again, we stress the fact that
$C^{N,\alpha}_B(]0,T[\times D)$ not only does contain $C^{\infty}(]0,T[\times D)$, but also
includes the standard H\"older space $C^{N,\alpha}(]0,T[\times D)$. Recent results on the local density assuming standard regularity of the coefficients were proved in \cite{BallyCar}, under local strong H\"ormander-type conditions, and in \cite{Pigato}, under local weak H\"ormander-type conditions for two-dimensional diffusions.

The proof of Theorem \ref{th:main} partially relies on some existence and regularity results  (see
\cite{DiFrancescoPascucci2} and \cite{Francesco} among others) obtained in a PDEs' context for the
fundamental solution and the Green functions of 
Kolmogorov operators, as well as on some Schauder estimates (see again \cite{Francesco}). However,
it is important to stress that the latter results alone are not enough to prove Theorem
\ref{th:main}. This is due to the fact that our structural assumptions on the generator of $X$
only hold on $D$, whereas there is no assumption on what is the behavior of the process outside
$D$. For this reason, it will be necessary to combine the PDE results mentioned above with some
probabilistic interlacing techniques and a crucial role will be played by the It\^o formula of
Theorem \ref{cor:ste101}.

More in detail, we adapt and customize the localization technique introduced in \cite{kusuoka-stroock}, Sec. 4.
We first prove a Feynman-Kac formula (Lemma \ref{lemm:ste3})
that allows to link the solutions of the Cauchy-Dirichlet problem for $\Lc$ on suitable cylinders
of $]0,T_0[\times D$ to the semigroup of the stopped diffusion.
As this step is based on the application of It\^o formula to the solution of the Cauchy-Dirichlet
problem and the latter is not of class $C^{2,\alpha}$ in the classical sense but only intrinsically, 
it is clear that the It\^o formula needed here is the one in Theorem \ref{cor:ste101} and not the
classical one. The proof of the existence of the local density can be then completed by following
closely the procedure employed by Kusuoka and Stroock, which makes use of a sequence of stopping
times that keep track of when the process exits and re-enters the domain $D$.

Once Part a) is proved, we need to depart from the latter procedure in order to prove part Part b)
and Part c). In particular, to obtain the intrinsic regularity of the local density
$\Gamma(t,x;T,y)$ w.r.t. the forward variables $(T,y)\in D$, it will be crucial to employ the
Schauder internal estimates for the solutions of $\Lc u=0$ proved in \cite{Francesco}, combined
with the Gaussian upper bounds for the Green function of ${\Lc}$ proved in the same reference.
Once we have proved that $\Gamma(t,x;\cdot,\cdot)\in C^{N,\alpha}_B(]t,T_0[\times D)$, then the
fact that $\Gamma(t,x;\cdot,\cdot)$ solves \eqref{eq:forward_kolm} simply follows because the
latter is satisfied by the transition probability kernel of $X$ in the distributional sense (see
Remark \ref{rem:kolm_forw_dist}).

To prove that $\Gamma(\cdot,\cdot;T,y)\in C^{N+2,\alpha}_B(]0,T[\times D)$ and solves
\eqref{eq:backward_kolm}, we first show that the same holds true for the function $(t,x)\to
E_{t,x}[\phi(X_T)]$ for any $\phi\in C_b(\mathbb{R}^d)$. This step is based again on a Feynman-Kac
formula and a crucial role is played one more time by the intrinsic It\^o formula of Theorem
\ref{cor:ste101}. Finally, Part c) follows by proving that the same properties hold true for any
bounded measurable function $\phi$ on $\mathbb{R}^d$. We remark that this last step is based on
the fact that $X$ actually enjoys the strong Feller property, namely the property in Assumption
\ref{assum:feller} extended to bounded measurable functions, which we prove in Lemma
\ref{lemm:strong_feller} by assuming the {standard} Feller property. Here we heavily rely again on
the Schauder estimates in \cite{Francesco}. We point out that the latter result, i.e. proving the
\emph{strong} Feller property starting from the \emph{standard} one, might enjoy an independent
interest as it generalizes some previous results obtained in \cite{schilling} under stronger
assumptions, basically existence and uniform boundedness of the global transition density.

We conclude this introduction mentioning that, 
as an application of our results it should be possible to prove sharp  Gaussian upper bounds 
for the local transition density $\G(t,x;T,y)$ and its derivatives.
The bounds would be analogous to those proved in \cite{kusuoka-stroock}, Theor. 4.5, for a wider
class of local generators, except that they would be explicit and valid under lower regularity
assumptions on the coefficients.

The rest of the paper is organized as follows. In Section \ref{sec:kolmogorov_operators} we recall
the definition of $B$-quasi-norm, the $B$-intrinsic H\"older spaces and the related intrinsic
Taylor formula. In Section \ref{sec:model} we give the precise definition of $\locdiff$-local
diffusion on $\R^d$ and prove Theorem \ref{cor:ste101}. In Section \ref{sec:model} we prove
Theorem \ref{th:main}. In Appendix \ref{app:pdes} we collect some useful PDE results for
$\Lc$-like operators, and in Appendix \ref{sec:app_markov} we recall some classic construction
procedures for Markov processes.


\section{Preliminaries: H\"older spaces and Taylor formula}\label{sec:kolmogorov_operators}
We recall the following useful characterization of Assumption \ref{ass:hypo}, proved in
\cite{LanconelliPolidoro1994}.
\begin{lemma}
Assumption \ref{ass:hypo} is fulfilled if and only if $B$ takes the block form
\begin{equation}\label{eq:B_blocks}
B=\left(
\begin{array}{ccccc}
\ast & \ast & \cdots & \ast & \ast \\ B_1 & \ast &\cdots& \ast & \ast \\ 0 & B_2 &\cdots& \ast&
\ast \\ \vdots & \vdots &\ddots& \vdots&\vdots \\ 0 & 0 &\cdots& B_r& \ast
\end{array}
\right)
\end{equation}
where $B_{j}$ is a $(p_j\times p_{j-1})$-matrix with full rank (equal to $p_{j}$) for
$j=1,\dots,r$, the $\ast$-blocks are arbitrary, $p_0\geq p_1\geq \cdots \geq p_r\geq 1$ and $p_0 +
p_1 + \cdots + p_r = d$.
\end{lemma}
We introduce the quasi-norm in $\R^{d}$ 
\begin{equation}\label{e7}
 |x|_B:=
 \sum_{j=0}^r \sum_{i=\bar{p}_{j-1}+1}^{\bar{p}_j} |x_{i}|^{2j+1}, \qquad 
{\bar{p}_j:= \sum_{k=0}^j p_k},\quad  {\bar{p}_{-1}:=0},
\end{equation}
that is homogeneous with respect to the dilations group
\begin{equation}\label{e7aa}
  D_{0}(\l)=\text{diag}\left(\l I_{p_{0}},\l^{3} I_{p_{1}},\dots,\l^{2r+1} I_{p_{r}}\right),\qquad
  \l>0.
\end{equation}
For any $(t,x)\in\Rdd$ and $i=1,\cdots,p_0$, we denote by
\begin{equation}\label{eq:def_curva_integrale_campo}
 e^{\d \p_{x_{i}}}(t,x)=(t,x+\delta e_i),\qquad
 e^{\d Y }(t,x)=(t+\delta,e^{\delta B}x),\qquad \delta>0,
\end{equation}
the integral curves of the vector fields $\p_{x_{1}},\cdots,\p_{x_{p_0}},Y$ starting at $(t,x)$.
{Here $e_i$ denotes the $i$-th element of the canonical basis of $\R^d$.}
Now, let $Q$ be a domain in $\Rdd$. For any $(t,x)\in Q$ we set
  $$\d_{(t,x)}:=\sup\left\{\bar{\d}\in\,]0,1]\mid e^{\d \p_{x_{1}}}(t,x),\cdots,e^{\d \p_{x_{p_0}}}(t,x),
  e^{\d Y}(t,x)\in Q\text{ for any }\d\in [-\bar{\d},\bar{\d}]\right\}.$$
If $V$ is compactly contained in $Q$ (hereafter we write $V\Subset Q$), we set
$\d_{V}=\inf\limits_{(t,x)\in V}\d_{(t,x)}.$
Note that $\d_{V}\in\,]0,1]$. 
\begin{definition}\label{def:intrinsic_alpha_Holder2}
Let $\a_1\in\,]0,1]$, $\a_2\in\, ]0,2]$ and $i=1,\cdots,p_0$. We say that $f\in
C_{\p_{x_{i}}}^{\alpha_1}( Q)$ and $g\in C_{Y}^{\alpha_2}( Q)$ if the following semi-norms are
finite
\begin{equation}\label{e12}
 \left\|f\right\|_{C_{\p_{x_{i}}}^{\alpha_1}(V)}:=\sup_{(t,x)\in V\atop 0<|\d|<\d_{V}}
 \frac{\left|f\left(e^{\delta \p_{x_{i}} }(t,x)\right)-
 f(t,x)\right|}{|\delta|^{\a_1}},
 \qquad \left\|g\right\|_{C_{Y}^{\alpha_2}(V)}:=\sup_{(t,x)\in V\atop 0<|\d|<\d_{V}}
 \frac{\left|g\left(e^{\delta Y }(t,x)\right)-
 g(t,x)\right|}{|\delta|^{\frac{\a_2}{2}}},
\end{equation}
for any $V\Subset Q$.
\end{definition}
We can now define the so-called $B$-H\"older spaces.
We point out that slightly different versions of such spaces were previously adopted in several works (see 
\cite{Francesco} and 
\cite{Manfredini} 
among others). Here we use the definition given \cite{PPP1}, which is basically  the one required
in order to prove an intrinsic Taylor formula where the remainder is in terms of the intrinsic
quasi-norm.
\begin{definition}\label{def:C_alpha_spaces}
Let $ Q$ a domain of $\Rdd$ and let $\a\in\,]0,1]$, then:
\begin{itemize}
  \item [i)] $f\in C^{0,\a}_{B}( Q)$ if $f\in C^{\a}_{Y}( Q)$ and $f\in C^{\a}_{{\p_{x_{i}}}}( Q)$ for any $i=1,\dots,p_{0}$;
  \item [ii)] $f\in C^{1,\a}_{B}( Q)$ if $f\in C^{1+\a}_{Y}( Q)$ and $\p_{x_{i}}f\in C^{0,\a}_{B}( Q)$ for any
  $i=1,\dots,p_{0}$;
  \item [iii)] {for} $n\in\mathbb{N}$ with $n\ge2$, $f\in C^{n,\a}_{B}( Q)$ if $Yf\in C^{n-2,\a}_{B}( Q)$ and
  $\p_{x_{i}}f\in C^{n-1,\a}_{B}( Q)$ for any
  $i=1,\dots,p_{0}$.
\end{itemize}
Moreover, for $f\in C^{n,\a}_{B}( Q)$ and $V\Subset Q$, we set
\begin{equation}
\|{f}\|_{C^{n,\a}_{B}( V)}:=\left\{
\begin{aligned}
&\|{f}\|_{C^{\a}_{Y}( V)}+\sum_{i=1}^{p_0} \|{f}\|_{C^{\a}_{\p_{x_{i}}}( V)},&& n=0\\
&\|{f}\|_{C^{\a+1}_{Y}( V)}+\sum_{i=1}^{p_0} \|{\partial_{x_i}f}\|_{C^{0,\a}_{B}( V)}, && n=1\\
&\|{Y f}\|_{C^{n-2,\a}_{B}( V)}+\sum_{i=1}^{p_0} \|{\partial_{x_i}f}\|_{C^{n-1,\a}_{B}( V)}, &&
n\geq 2.
\end{aligned}
\right.
\end{equation}
If $f\in C^{n,\a}_{B}(Q)$ and has compact support then we write $f\in
C_{0,B}^{n,\alpha}\left(Q\right)$.\end{definition}

The next result was proved in \cite{PPP1} in the particular case when the $\ast$-blocks  in
\eqref{eq:B_blocks} are null and then extended to the general case in \cite{PagPig}.
\begin{theorem}\label{th:main_tay}
Let $ Q$ be a domain of $\Rdd$, $\alpha\in\, ]0,1]$ and $n\in\Nzero$. If $f\in C^{n,\a}_{B}( Q)$
then we have:
\begin{enumerate}
\item[1)] there exist 
\begin{equation}\label{eq:maintheorem_part1}
 Y^k \partial_x^{\beta}f\in C^{n-2k-[\beta]_B,\alpha}_{B}( Q),
 \qquad {0\leq 2 k + [\beta]_B \leq n},
\end{equation}
where $[\b]_B$ denotes the height of the multi-index $\b$ defined as
\begin{equation}
 [\b]_B:  =   \sum\limits_{j=0}^r \sum_{i=\bar{p}_{j-1}+1}^{\bar{p}_j} (2 j +1) \b_i;
\end{equation}

\item[2)]
for any $(t_{0},x_{0})\in Q$, there exist two bounded domains $U,V$, such that $(t_{0},x_{0})\in U\subseteq V\subseteq Q$ 
and
\begin{equation}\label{eq:estim_tay_n_loc}
 \left|f(t,x)-\mathbf{T}^{(n)}_{(s,y)} f(t,x)\right|\le c_{B,U
 } \|f\|_{C^{n,\a}_{B}( V
 )}
 \Big( |s-t|^{1/2}+\big|y-e^{(s-t)B}x\big|_B\Big)^{n+\a},\qquad (t,x),(s,y) \in U,
\end{equation}
where $c_{B,U}$ is a positive constant 
and $\mathbf{T}^{(n)}_{(s,y)}$ is the {$n$-th order intrinsic Taylor polynomial of $f$ centered at
$(s,y)
$} given by 
\begin{equation}\label{eq:ste_Tay_pol}
 \mathbf{T}^{(n)}_{(s,y)} f(t,x)=  \sum_{\substack{k\in\N_0,\, \beta\in\N_0^d\\0\leq 2 k + [\beta]_B \leq n}}\frac{1}{k!\,\beta!}
 \big( Y^k \partial_{y}^{\beta}f(s,y)\big) (t-s)^k\big(x-e^{(t-s)B}y  \big)^{\beta},\qquad
 (t,x)\in\Rdd. 
\end{equation}
\end{enumerate}
\end{theorem}
\begin{corollary}\label{cor_tay_int}
If $f\in C^{n,\a}_{0,B}( Q)$, then \eqref{eq:estim_tay_n_loc} holds true with {$U=\text{\rm
supp}(f)$} and $V= Q$.
\end{corollary}


\section{Local diffusions and {intrinsic} It\^o formula}\label{sec:model}
For a given $T_{0}>0$ we consider a continuous $\R^d$-valued strong Markov process  $X=(X_{t})_{t\in[0,T_0[}$
(in the sense of \cite{FriedmanSDE1
} as it is recalled in Appendix \ref{sec:app_markov}) with transition probability function $p=p(t,x;T,d\x)$, defined on a space
$$
\left(\O,\F,(\F_{T}^{t})_{0\le t\le T< T_{0}},(P_{t,x})_{0\le t< T_{0},x\in\R^d}\right).
$$
For any bounded Borel measurable function $\phi$, we denote by
\begin{equation}\label{ae67}
   E_{t,x}\left[\phi(X_{T})\right]:=(\TT_{t,T}\phi)(x):=\int_{\R^{d}}p(t,x;T,d\x)\phi(\x),\qquad 0\le t<T\le
   T_{0},\ x\in \R^{d},
\end{equation}
the $P_{t,x}$-expectation and the semigroup associated with the transition probability function
$p$, respectively (cf. Chapter 2.1 in \cite{FriedmanSDE1}). Hereafter 
we also fix a domain $D$, that is an open and connected subset of $\R^d$.

\begin{notation}
For a given function $f(t,T)$ with $T\in\, ]0,T_0[$ and $t\in[0,T[$, we set
\begin{align}
 \lim_{T-t\rightarrow 0^+} f(t,T) &:= \lim_{h\rightarrow 0^+} f(t,t+h) =\lim_{h\rightarrow 0^+} f(t-h,t), \qquad t\in[0,T_0[,
\end{align}
when the second and the third limits exist and coincide with each other.
\end{notation}
The following two sets of limits will be used to give the definition of {local diffusion}. 
\begin{itemize}
\item[{[\bf Lim-i)]}]
For any $t\in[0,T_{0}[$, $\d>0$, 
  and $H$ compact subset of $D$, there exist the limits
\begin{align}\label{ae51b}
  \lim_{T-t\rightarrow 0^+}\!\!\!\!\int\limits_{\{|\x-x|>\d 
  \}\cap H}\frac{p(t,x;T,d\x)}{T-t}
  &=0  ,  \qquad \text{uniformly w.r.t. $x\in\R^d$}, \\
   \label{ae51}
  \lim_{T-t\rightarrow 0^+}\!\!\!\!\int\limits_{|\x-x|>\d 
  }\frac{p(t,x;T,d\x)}{T-t}
  &=0, \qquad \text{uniformly w.r.t. $x\in H$.}
\end{align}

\item[{\bf [Lim-ii)]}]
For any $t\in[0,T_{0}[$, $\d>0$, 
  and $H$ compact subset of $D$, 
  and for any $i=1,\cdots,d$, there exist the limits
  \begin{align}
    \lim_{T-t\rightarrow 0^+}\!\!\!\!\int\limits_{|\x-x|<\d 
  }\!\!\!\!(\x-x)_{i}\frac{p(t,x;T,d\x)}{T-t}
  &=  \begin{cases}
\ac_{i}(t,x)+(Bx)_i & \text{if }i=1,\cdots,p_0
 \\
(Bx)_i &  \text{if } i=p_0+1,\cdots,d 
  \end{cases},\label{ae51c_bis}
  \\ \label{ae51d_bis}
  \lim_{T-t\rightarrow 0^+}\!\!\!\!\int\limits_{|\x-x|<\d 
  }\!\!\!\!(\x-x)_{i}(\x-x)_{j}\frac{p(t,x;T,d\x)}{T-t}
 & = \begin{cases}
\ac_{ij}(t,x) & \text{if } i,j = 1,\cdots, p_0 
\\
0 & \text{if $i$ or $j= p_0+1,\cdots,d$}  
  \end{cases},
\end{align}
uniformly w.r.t. $x\in H$, for some $B$ as in \eqref{eq:B_blocks} and $a_{ij},a_i\in
L^{\infty}_{\rm loc}([0,T_0[\times D)$.
\end{itemize}
\begin{definition}\label{def:local_diffusion}
Let 
$\Ac_t$ an operator as in \eqref{ae75bis}. We say that $X$ is a \emph{local diffusion 
generated by $\Ac_t$ on $D$} (an \emph{$\locdiff$-local diffusion
} in short) if
{\bf[Lim-i)]} and {\bf[Lim-ii)]} hold. In case 
they hold with $D = \R^d$ then we call $X$ a \emph{global diffusion 
generated by $\Ac_t$} (an
\emph{$\Ac_t$-global diffusion
} in short).
\end{definition}
The following proposition is useful for the applications because several models are defined in
terms of solutions to stochastic differential equations. It shows that (stopped) solutions of SDEs
are local diffusions in the sense of Definition \ref{def:local_diffusion}. 
\begin{remark}
Since we are dealing with stopping times, we point that we did not impose any right-continuity assumption on the filtrations $\F^t$. 
This is justified by the fact that, in the next proposition as well as in the rest of the paper, we will only consider hitting times of closed sets, which appear to be stopping times even if the the filtration is not right-continuous (see \cite{FriedmanSDE1}, Theorem 2.2 p. 25).
\end{remark}
\begin{proposition}\label{la6}
Let $\left(X_{t}\right)_{t\in[0,T_0[}$ be a continuous Markov process defined as
$X_{t}=\hat{X}_{t\wedge \t}$, where:
\begin{itemize}
  \item[i)] $\hat{X}$ is a solution of the SDE
\begin{align}\label{eq:sde}
 d\hat{X}_{t}=\m(t,\hat{X}_{t})dt+\s(t,\hat{X}_{t})dW_{t}
 \end{align}
where $W$ is a $n$-dimensional Brownian motion and the coefficients $\m$ and $\s$ are continuous
and as in \eqref{eq:ste105}-\eqref{eq:ste103};
  \item[ii)] $\t$ is the first exit time of $\hat{X}$ from a domain $D'$ containing $D$.
\end{itemize}
Then $X$ is a $\locdiff$-local diffusion on $\R^d$ in the sense of Definition
\ref{def:local_diffusion}.
\end{proposition}
\begin{proof}
The statement is a particular case of Lemma 2.3 in \cite{PPTaylor}, which proves that {\bf[Lim-i)]} and {\bf[Lim-ii)]} hold for the kernel of $X$.
\end{proof}
We have the first key result.
\begin{proposition}\label{prop:ste1}
Let Assumptions \ref{ass:hypo} be in force. Then $X$ is an {$\locdiff$-local diffusion on $\R^d$}
if and only if {\bf[Lim-i)]} holds and for any $t\in[0,T_{0}[$, $\d>0$, 
  and $H$ compact subset of $D$, we have
\begin{align}
  &\lim_{T-t\rightarrow 0^+}\!\!\!\!\int\limits_{|\x-x|<\d 
  }\!\!\!\!\big(\x-e^{(T-t)B}x\big)_{i}\frac{p(t,x;T,d\x)}{T-t}
  =\ac_{i}(t,x),\label{ae51c}\\ \label{ae51d}
  &\lim_{T-t\rightarrow 0^+}\!\!\!\!\int\limits_{|\x-x|<\d 
  }\!\!\!\!\big(\x-e^{(T-t)B}x\big)_{i}\big(\x-e^{(T-t)B}x\big)_{j}\frac{p(t,x;T,d\x)}{T-t}
  =\ac_{ij}(t,x),\\
  &\limsup_{T-t\rightarrow 0^+}\!\!\!\!\int\limits_{|\x-x|<\d 
  } 
\big|  \x-e^{(T-t)B}x\big|^2_B \frac{p(t,x;T,d\x)}{T-t}
  <\infty,  \label{ae51e}
\end{align}
for any $i,j=1,\cdots,p_0$, uniformly w.r.t. $x\in H$, for some $B$ as in \eqref{eq:B_blocks} and $a_{ij},a_i\in
L^{\infty}_{\rm loc}([0,T_0[\times D)$.
\end{proposition}
The following lemma 
formalizes the fact that $\Ac_t$ as in \eqref{ae75} is the generator of $X$ on $[0,T_0[\times D$.
\begin{proposition}\label{la5}
Let $X$ be a $\locdiff$-local diffusion on $\R^d$ and Assumption \ref{ass:hypo} be in force. Then,
for any {$\phi \in C_{0}(]0,T_{0}[\times D)$
} and 
$f\in C_{0,B}^{2
}\left(]0,T_0[\times D\right)$
we have 
\begin{align}
 \label{ae53}
 &\lim_{T-t\rightarrow 0^+} 
 \sup_{x\in\Rd} \big| {\big(\TT_{t,T}\phi(T,\cdot)\big)(x)-\phi(t,x)}\big|
 =0,\\
\label{ae71}
 &\lim_{T-t\rightarrow 0^+}
 \sup_{x\in\Rd}\bigg| \frac{\big(\TT_{t,T}f(T,\cdot)\big)(x)-f(t,x)}{T-t}-
 \Lc f(t,x)\bigg|
 =0,
\end{align}
for any $t\in]0,T_0[$. Moreover, for any $0< t< T< T_{0}$ and $x\in \R^d$, it holds that
\begin{equation}\label{ae62}
 \frac{d}{dT}\big(\TT_{t,T}f(T,\cdot)\big)(x)=\TT_{t,T}\big(
 \Lc f(T,\cdot)\big)(x).
\end{equation}
\end{proposition}

\begin{remark}\label{rem:new}
In \cite{PPTaylor} it was already proven a weaker form of Proposition \ref{la5} and Theorem \ref{cor:ste101}, which applies to a wider class of local generators $\Ac_t$. When applied to this particular framework, it yields that if 
$X$ is a $\locdiff$-local diffusion on $\R^d$, then 
\eqref{ae53}, \eqref{ae71}, \eqref{ae62} and thus also \eqref{ae41} and \eqref{eq:quad_var}, hold true for a smaller class of functions, namely $f\in C^{2}_{0}\left(]0,T_{0}[\times D\right)$.
\end{remark}

\begin{remark}\label{rem:kolm_forw_dist}
By Proposition \ref{cor:ste101} it is clear that, for any $(t,x)\in\, ]0,T_0[\times \R^d$,
$p(t,x;\cdot,\cdot)$ satisfies
\begin{equation}\label{ae44}
\Lc^*  u = 0 \qquad \text{on $]t,T_{0}[\times {D}$},
\end{equation}
in the sense of distributions, with $
\Lc^*$ being the formal adjoint of $
\Lc$, i.e.
\begin{equation}\label{ae42_bis}
  \int_{t}^{T_{0}}\int_{{D}}p(t,x;T,d\x)
  \Lc f(T,\x)d T = 0, \qquad f\in C_0^{\infty}(]t,T_0[\times D).
\end{equation}
\end{remark}

\subsection{Proofs 
}\label{subsec:proof_theorem_density}

In this section we prove Proposition \ref{prop:ste1}, Proposition \ref{la5} and Theorem \ref{cor:ste101}, respectively, the proof of each being based on the previous one.
\subsubsection{Proof of Proposition \ref{prop:ste1}}
The proof of Proposition \ref{prop:ste1} is preceded by the following
\begin{lemma}\label{lem:ste01}
Under the hypothesis {\bf[Lim-i)]}, 
for any $H\Subset D$ and for any $0\leq t < T<T_0$, $\d<\bar{\d}:=\text{\rm dist}(H,\p D)$, we
have
\begin{equation}\label{ae51_stronger}
\lim_{T-t\rightarrow 0^+}\!\!\!\!\int\limits_{\{|\x-x|\geq\d 
  \}\cap H}\frac{p(t,x;T,d\x)}{(T-t)^m}
  =0,\qquad m\geq 1,
\end{equation}
uniformly w.r.t. $x\in H$.
\end{lemma}
\begin{proof}
Let $
\phi^{(x)} \in C_{0}^{\infty}( D)$ be a family of functions with partial derivatives uniformly bounded w.r.t. $x\in H$, and such that $\phi^{(x)}(x)=0$ and
\begin{equation}\label{eq:ste106_quat}
\phi=\phi^{(x)} = 1\quad \text{on }H\setminus B(x,\d).
\end{equation}
Note that Remark \ref{rem:new} gives
\begin{equation}
\phi(X_T)  =\phi(X_{t})+ \int_{t}^{T}  \Acc_{s}  \phi(X_{s})ds + M^t_T,
\end{equation}
where $M^t$ is an $\F^{t}$-martingale under $P_{t,x}$ with 
\begin{equation}\label{eq:quad_var_ter}
E_{t,x}\big[ |M^t_T|^2 \big] = E_{t,x}\bigg[ \int_t^T   \sum_{k,l=1}^{p_0}  
 a_{kl}(s,X_s)  \big(\partial_{x_k} \phi(X_s)\big)\big(\partial_{x_l} \phi(X_s)\big)  \dd s\bigg].
\end{equation}
Thus we obtain
\begin{align}
\int\limits_{\{|\x-x|\geq\d 
  \}\cap H}{p(t,x;T,d\x)} & \leq E_{t,x} \big[ \big( \phi(X_T) \big)^{2(m+1)}  \big]  \leq 2^{2m+1}
 E_{t,x} \bigg[ \bigg| \int_{t}^{T}  \Acc_{s}  \phi(X_{s})ds \bigg|^{2(m+1)} + | M^t_T |^{2(m+1)}  \bigg]\\
& \leq 2^{2m+1} \bigg(
 E_{t,x} \bigg[ \bigg| \int_{t}^{T}  \Acc_{s}  \phi(X_{s})ds \bigg|^{2(m+1)} \bigg]+ \bigg( \frac{2(m+1)}{2m + 1} \bigg)^{2(m+1)} E_{t,x}\big[ 
 |M^t_T|^2   \big]^{m+1}     \bigg),
\end{align}
where we used \cite[Lemma 3.8, p. 71]{FriedmanSDE1} and Jensen's inequality in the last step. 
Eventually, \eqref{ae51_stronger} stems from \eqref{eq:quad_var_ter} combined with the fact that $\phi\in C_{0}^{\infty}( D)$ and the coefficients of $\Ac$ are in $L^{\infty}_{\rm loc}([0,T_0[\times D)$. 
\end{proof}

We are now in the position to prove Propostion \ref{prop:ste1}.

\begin{proof}[Proof of Proposition \ref{prop:ste1}]${}$

\vspace{3pt}\noindent
\underline{Part 1}: \emph{if}.
We assume {\bf[Lim-i)]} to be satisfied together with the limits \eqref{ae51c}-\eqref{ae51d}-\eqref{ae51e}, and prove that {\bf[Lim-ii)]} holds. Let $t\in[0,T_{0}[$, $\d>0$, 
and $H$ compact subset of $D$ be fixed. All the following limits are uniform w.r.t. $x\in H$.

We first prove \eqref{ae51c_bis}.
For any $i=1,...,d$, it holds that : 
\begin{align}
\lim_{T-t\rightarrow 0^+}\!\!\!\!\int\limits_{|x-\x|<\delta}\!\!\!\!(\x-x)_{i}\frac{p(t,x;T,d\x)}{T-t}&=\lim_{T-t\rightarrow 0^+}\!\!\!\!\int\limits_{|x-\x|<\delta}\!\!\!\! \big(\x -e^{(T-t)B}x +e^{(T-t)B}x -x \big)_{i}\frac{p(t,x;T,d\x)}{T-t}\\
&=\lim_{T-t\rightarrow 0^+}\!\!\!\!\int\limits_{|x-\x|<\delta}\!\!\!\!\big(\x-e^{(T-t)B}x\big)_{i}\frac{p(t,x;T,d\x)}{T-t}+(Bx)_i.\label{eq:ste301}
\end{align}
Here we used the property
\begin{equation}\label{ae73}
\lim_{T-t\rightarrow 0^+}\int\limits_{|x-\x|<\delta}\!\!\!\!p(t,x;T,d\x)=1,
\end{equation}
which follows from 
{\bf[Lim-i)]}. Now, for $i=1,\cdots,p_0$, \eqref{ae51c_bis} stems from \eqref{ae51c} and \eqref{eq:ste301}. On the other hand, for $i=p_0+1,\cdots,d$, note that condition \eqref{ae51e} is equivalent to
\begin{equation}\label{equiv}
\limsup_{T-t\rightarrow 0^+}\!\!\!\!\int\limits_{|\x-x|<\d 
  } 
\big|  \big(\x-e^{(T-t)B}x\big)_i\big|^{\frac{2}{q_i}}\frac{p(t,x;T,d\x)}{T-t}
  <\infty
\end{equation}
for a certain $q_i\geq 3$. Therefore, fixing $\varepsilon>0$, for any $\rho\in\, ]0,\delta]$ 
we obtain
\begin{align}\label{delta}
&\limsup_{T-t\rightarrow 0^+}\int\limits_{|\x-x|<\d}\big|  \big(\x-e^{(T-t)B}x\big)_i\big|^{\frac{2}{q_i}+\varepsilon}\frac{p(t,x;T,d\x)}{T-t}\\
&= \limsup_{T-t\rightarrow 0^+}\int\limits_{|\x-x|<\rho}\big|  \big(\x-e^{(T-t)B}x\big)_i\big|^{\frac{2}{q_i}+\varepsilon}\frac{p(t,x;T,d\x)}{T-t} + \int\limits_{\rho\leq |\x-x|<\d}\big|  \big(\x-e^{(T-t)B}x\big)_i\big|^{\frac{2}{q_i}+\varepsilon}\frac{p(t,x;T,d\x)}{T-t} \\
&\leq \limsup_{T-t\rightarrow 0^+}\int\limits_{|\x-x|<\rho}\big|  \big(\x-e^{(T-t)B}x\big)_i\big|^{\frac{2}{q_i}+\varepsilon}\frac{p(t,x;T,d\x)}{T-t} +\big( \d + |x| \big|I_d-e^{(T-t)B}\big|\big)^{\frac{2}{q_i}+\varepsilon} \int\limits_{\rho\leq |\x-x|<\d} \frac{p(t,x;T,d\x)}{T-t}
\intertext{(by \eqref{ae51})}
&= \limsup_{T-t\rightarrow 0^+}\int\limits_{|\x-x|<\rho}\big|  \big(\x-e^{(T-t)B}x\big)_i\big|^{\frac{2}{q_i}+\varepsilon}\frac{p(t,x;T,d\x)}{T-t}\\
&\leq\limsup_{T-t\rightarrow 0^+}\,\big( \rho + |x| \big|I_d-e^{(T-t)B}\big|\big)^{\varepsilon}\int\limits_{|\x-x|<\rho}\big|  \big(\x-e^{(T-t)B}x\big)_i\big|^{\frac{2}{q_i}}\frac{p(t,x;T,d\x)}{T-t}\\
&\leq \rho^{\varepsilon} \limsup_{T-t\rightarrow 0^+}\int\limits_{|\x-x|<\d}\big|  \big(\x-e^{(T-t)B}x\big)_i\big|^{\frac{2}{q_i}}\frac{p(t,x;T,d\x)}{T-t},
\end{align}
which in turn implies
\begin{equation}\label{vanishing}
\limsup_{T-t\rightarrow 0^+}\!\!\!\!\int\limits_{|\x-x|<\d}\big|  \big(\x-e^{(T-t)B}x\big)_i\big|^{\frac{2}{q_i}+\varepsilon}\frac{p(t,x;T,d\x)}{T-t}=0,
\end{equation}
and thus, as $\eps$ is arbitrary,
\begin{equation}
\lim_{T-t\rightarrow 0^+}\!\!\!\!\int\limits_{|x-\x|<\delta}\!\!\!\!(\x-e^{(T-t)B}x)_{i}\frac{p(t,x;T,d\x)}{T-t}=0.
\end{equation}
The latter, combined with \eqref{eq:ste301}, proves \eqref{ae51c_bis} for $i=p_0+1,\cdots,d$.

We now prove \eqref{ae51d_bis}. By using \eqref{ae51c_bis} it is straightforward to show that
\begin{equation}\label{eq:ste303}
\lim_{T-t\rightarrow 0^+}\!\!\!\!\int\limits_{|\x-x|<\d}
\!\!\!\!\big(\x-x\big)_{i}\big(\x-x\big)_{j}\frac{p(t,x;T,d\x)}{T-t} =\lim_{T-t\rightarrow 0^+}\!\!\!\! \int\limits_{|\x-x|<\d}
\!\!\!\!\big(\x-e^{(T-t)B}x\big)_{i}\big(\x-e^{(T-t)B}x\big)_{j}\frac{p(t,x;T,d\x)}{T-t}.
\end{equation}
Now, for $i,j=1,...,p_0,$ \eqref{ae51d_bis} simply stems from \eqref{ae51d}.
In the case $i>p_0$ or $j>p_0$, 
by the Cauchy-Schwarz inequality we get
\begin{align}
\int\limits_{|\x-x|<\d}
\!\!\!\!\big| \big(\x-e^{(T-t)B}x\big)_{i}\big(\x-e^{(T-t)B}x\big)_{j}\big| \frac{p(t,x;T,d\x)}{T-t}&\leq
\Bigg(\,\int\limits_{|\x-x|<\d}
\!\!\!\!\big(\x-e^{(T-t)B}x\big)^2_{i}\frac{p(t,x;T,d\x)}{T-t}\Bigg)^{\frac{1}{2}}\\
&\quad \times\Bigg(\,\int\limits_{|\x-x|<\d}
\!\!\!\!\big(\x-e^{(T-t)B}x\big)^2_{j}\frac{p(t,x;T,d\x)}{T-t}\Bigg)^{\frac{1}{2}},
\end{align}
and the limits of the right-hand side integrals are both finite, at lest one of each being zero. In fact, by \eqref{ae51d} and \eqref{vanishing} we obtain
\begin{equation}
\int\limits_{|\x-x|<\d}
\!\!\!\!\big(\x-e^{(T-t)B}x\big)^2_{k}\frac{p(t,x;T,d\x)}{T-t} =\begin{cases}
a_{kk}(t,x) &  \text{if } k=1,\cdots,p_0 \\
0    & \text{if } k=p_0+1, \cdots, d.
  \end{cases}
\end{equation}
Therefore, we can conclude that for $i>p_0$ or $j>p_0$
\begin{equation*}
\lim_{T-t\rightarrow 0^+}\!\!\!\!\int\limits_{|\x-x|<\d}
\!\!\!\!\big(\x-e^{(T-t)B}x\big)_{i}\big(\x-e^{(T-t)B}x\big)_{j}\frac{p(t,x;T,d\x)}{T-t}=0,
\end{equation*}
which, combined with \eqref{eq:ste303}, yields \eqref{ae51d_bis}.

\vspace{3pt}\noindent
\underline{Part 2}: \emph{only if}. We now assume 
{\bf[Lim-i)]}-{\bf[Lim-ii)]} and prove the limits \eqref{ae51c}-\eqref{ae51d}-\eqref{ae51e} to be true. 
Fix $t\in[0,T_{0}[$, $\d>0$, 
and $H$ compact subset of $D$. Again, all the following
limits are uniform w.r.t. $x\in H$.

The limits \eqref{ae51c} and \eqref{ae51d} stem again from \eqref{eq:ste301}-\eqref{eq:ste303} with $i,j=1,\cdots,p_0$. Thus to conclude we only need to prove \eqref{ae51e}.
We remark that, by \eqref{ae51}, 
it is not restrictive to assume $\d<\text{\rm dist}(H,\p D)$. By \eqref{ae51d} and by Jensen's inequality, it is sufficient to prove
\begin{equation}\label{eq:newlimit}
\limsup_{T-t\rightarrow 0^+}\frac{1}{T-t} E_{t,x}\Big[\big|\big(X_T-e^{(T-t)B}x\big)_j\big|
\caratt_{\{|X_T -x| <\delta\}}   \Big]^{\frac{2}{2i+1}} <\infty,
\end{equation}
for any fixed $i\in\{ 1,\cdots,r\}$ and 
$j\in\{\sum_{k=0}^{i-1}p_k+1,\cdots,d\}$. We prove \eqref{eq:newlimit} in two different steps:

\noindent\emph{Step 1.} We prove that
\begin{align}
\limsup_{T-t\to 0^+}\frac{1}{T-t}E_{t,x}\Big[\big|&\big(X_T-e^{(T-t)B}x\big)_j\big|
\caratt_{\{|X_T -x| <\delta\}}   \Big]^{\frac{2}{2i+1}}\\
& \leq  \limsup_{T-t\to 0^+}\frac{1}{T-t} \bigg(   \int_t^{T}  E_{t,x}\Big[ \big| \big\langle B^{(j)}, X_s - e^{(s-t)B} x\big\rangle\big|  \caratt_{\{|X_s -x| <\delta\}}  \Big]   
\dd s   \bigg)^{\frac{2}{2i+1}}  ,
\label{eq:ste107}
\end{align}
for any $i\geq 1$ and $j=\{{p}_0+1,\cdots,d\}$. Here and further on, $B^{(j)}$ denotes the $j$-th row of $B$. 

Let $
\phi^{(t,x)}_j \in C_{0}^{\infty}([t,T_0[\times D)$, $j={p}_0+1,\cdots,d$, be a family of
functions with partial derivatives uniformly bounded w.r.t. $(t,x)\in [0,T_0[\times H$ and such
that
\begin{equation}\label{eq:ste106_bis}
\phi_j(s,\xi)=\phi^{(t,x)}_j(s,\x) = \x_j  -\big(e^{(s-t)B}x\big)_j,\qquad 
|\x - x|<\d.
\end{equation}
Note that we have
\begin{equation}\label{eq:ste106}
\begin{cases}
\partial_s \phi_j(s,\x) = - \big\langle B^{(j)}, e^{(s-t)B}x\big\rangle\\
\partial_{\xi_j}\phi_j(s,\x) = 1\\
\partial_{\xi_k}\phi_j(s,\x) = 0  \quad \text{for } k\neq j\\
\partial_{\xi_k \xi_l}\phi_j(s,\x)= 0
\end{cases},\qquad s\in [t,T_0[,\ |\xi - x|<\d. 
\end{equation}
By Remark \ref{rem:new} we have
%
\begin{align}
\hspace{-5pt} \phi_j(T,X_T) & =\phi_j(t,X_{t})+ \int_{t}^{T} (\partial_s + \Acc_{s}) \phi_j(s,X_{s})ds + M^t_T
\intertext{(by \eqref{eq:ste106_bis}-\eqref{eq:ste106})}
\hspace{-10pt}&\hspace{-10pt} =
\phi_j(t,X_t)+ \int_{t}^{T}\!\! \big\langle B^{(j)}, X_s -e^{(s-t)B}x \big\rangle \caratt_{\{|X_s -x|<\delta
\}} ds + \int_{t}^{T} (\partial_s+\Acc_{s}) \phi_j(s,X_{s})\caratt_{\{|X_s -x|\geq\delta
\}}  ds + M^t_T,\hspace{-5pt}
\end{align}
where $M^t$ is an $\F^{t}$-martingale under $P_{t,x}$ with 
\begin{equation}\label{eq:quad_var_bis}
E_{t,x}\big[ |M^t_T|^2 \big]= E_{t,x}\bigg[ \int_t^T   \sum_{k,l=1}^{p_0}  
 a_{kl}(s,X_s)  \big(\partial_{x_k} \phi_j(s,X_s)\big)\big(\partial_{x_l} \phi_j(s,X_s)\big)  \dd s\bigg].
\end{equation}
Therefore, for any $i\geq 1$ and $j={p}_0+1,\cdots,d$ one obtains
\begin{equation}
E_{t,x}\Big[\big|\big(X_T-e^{(T-t)B}x\big)_j\big| \caratt_{\{|X_T -x| <\delta\}}   \Big]^{\frac{2}{2i+1}}  \leq 
E_{t,x}\big[\big|\phi_j(T,X_T) \big|\big]^{\frac{2}{2i+1}}\leq  
\sum_{k=1}^3 I^{\frac{2}{2i+1}}_k(t,x;T) ,
\end{equation}
where
\begin{align}
I_1(t,x;T) &= E_{t,x}\bigg[\bigg|   \int_{t}^{T} (\partial_s + \Acc_{s})
\phi_j(s,X_{s})\caratt_{\{|X_s -x|\geq\delta\}}  ds  \bigg|\bigg], \qquad I_2(t,x;T) =
E_{t,x}\big[\big|M^t_T  \big|\big], \\ I_3(t,x;T) &= E_{t,x}\bigg[\bigg|  \int_t^{T} \big\langle
B^{(j)}, X_s -e^{(s-t)B}x \big\rangle \caratt_{\{|X_s -x|<\delta\}}   \dd s  \bigg|\bigg] ,
\end{align}
where we used that $\phi_j(t,X_t) = 0$ $P_{t,x}$-almost surely.
Since 
the coefficients of $\Ac$ are locally bounded on $[0,T_0[\times D$ and $\phi_j\in C_0^{\infty}([t,T_0[\times D)$, we obtain
\begin{align}
I_1(t,x;T) &   \leq  
C  \int_t^{T} E_{t,x}\big[  \caratt_{\{|X_s -x|\geq\delta\}}\caratt_{\{X_s\in \text{supp} (\phi_j)\}} \big]  \dd s  =  C  \int_t^{T} (s-t)^{\frac{2i+1}{2}}\!\!\!\! \int\limits_{(|\x -x|\geq\delta) \cap \text{supp}(\phi_j) 
  }\frac{p(t,x;s,d\z)}{(s-t)^{\frac{2i+1}{2}}}  \dd s
  \intertext{(by \eqref{ae51_stronger} with $m=\frac{2i+1}{2}$)}
& \leq  C   \int_t^{T} (s-t)^{\frac{2i+1}{2}}   \dd s \leq C (T-t)^{\frac{2i+1}{2}+1},
\end{align}
for any $i\geq 1$, 
which proves
\begin{equation}
\lim_{T-t\to 0^+}\frac{1}{T-t}  I^{\frac{2}{2i+1}}_1(t,x;T) = 0,\qquad i\geq 1.
\end{equation}
Similarly, Jensen's inequality 
yields 
\begin{align}
I_2(t,x;T)& \leq E_{t,x}\big[|M^t_T  |^2\big]^{\frac{1}{2}} 
\intertext{(by \eqref{eq:quad_var_bis} combined with \eqref{eq:ste106})}
&=E_{t,x} \bigg[
 \int_t^T  \caratt_{\{|X_s -x|\geq\delta\}} \sum_{k,l=1}^{p_0}  
 a_{kl}(s,X_s)  \big(\partial_{x_k} \phi_j(s,X_s)\big)\big(\partial_{x_l} \phi_j(s,X_s)\big)
\dd s \bigg]^{\frac{1}{2}} \intertext{(by using again the local boundedness of $a_{kl}$ on
$[0,T_0[\times D$ and $\phi_j\in C_0^{\infty}([0,T_0[\times D)$)} &\leq C E_{t,x}\bigg[ \int_t^T
\caratt_{\{|X_s -x|\geq\delta\}}\caratt_{\{X_s\in \text{supp} (\phi_j)\}}  \dd s
\bigg]^{\frac{1}{2}},
\end{align}
and by proceeding as we did to estimate $I_1(t,x;T)$ it easily follows that
\begin{equation}
\lim_{T-t\to 0^+}\frac{1}{T-t}  I^{\frac{2}{2i+1}}_2(t,x;T) = 0,\qquad i\geq 1.
\end{equation}
Finally, 
\begin{equation}
\limsup_{T-t\to 0^+}\frac{1}{T-t}  I^{\frac{2}{2i+1}}_3(t,x;T) \leq  \limsup_{T-t\to 0^+}\frac{1}{T-t} \bigg(   \int_t^{T}  E_{t,x}\Big[ \big| \big\langle B^{(j)}, X_s - e^{(s-t)B} x\big\rangle\big|\caratt_{\{|X_s -x| <\delta\}}  \Big]    
\dd s   \bigg)^{\frac{2}{2i+1}}  
\end{equation}
yields \eqref{eq:ste107} for any $i\geq 1$.

\noindent\emph{Step 2.} We prove \eqref{eq:newlimit} by induction on $i$. To start the inductive
procedure, set $i=0$ and prove \eqref{eq:newlimit} for any $j\in\{ 1,\cdots,d  \}$. If
$j\in\{1,\dots,\bar{p}_0\}$, then \eqref{eq:newlimit} stems from $\eqref{ae51d}$ by applying
Jensen's inequality. If $j\in\{\bar{p}_0+1,\cdots,d\}$, then \eqref{eq:newlimit} follows trivially
from \eqref{eq:ste107} by observing that $$ \big| \big\langle B^{(j)}, X_s - e^{(s-t)B}
x\big\rangle\big| \caratt_{\{|X_s -x| <\delta\}} $$ is uniformly bounded.

Set now $\bar{i}\in \{0,\cdots,r-1\}$, assume \eqref{eq:newlimit} true for $i=\bar{i}$ and $j\in\{
\sum_{k=0}^{\bar{i}-1} p_k+1,\cdots, d\}$ and prove it true for $i=\bar{i}+1$ and $j\in\{
\sum_{k=0}^{\bar{i}} p_k+1,\cdots, d\}$. Now, by the block structure of $B$ \eqref{eq:B_blocks},
we obtain
\begin{align}
\int_t^T E_{t,x}\Big[ \big| \big\langle B^{(j)}, X_s - e^{(s-t)B} x\big\rangle\big|
\caratt_{\{|X_s -x| <\delta\}} \Big] \dd s&\leq C \sum_{k=\sum_{k=0}^{\bar{i}-1} p_k+1}^d \int_t^T
E_{t,x}\Big[ \big|\big( X_s - e^{(s-t)B} x\big)_k \big|  \caratt_{\{|X_s -x| <\delta\}} \Big]
\dd s \intertext{(by inductive hypothesis)} &\leq C \int_t^T (s-t)^{\frac{2\bar{i}+1}{2}}  \dd s
\leq C (T-t)^{\frac{2(\bar{i}+1)+1}{2}},
\end{align}
which, combined with \eqref{eq:ste107}, yields exactly \eqref{eq:newlimit} for $i=\bar{i}+1$ and concludes the proof.

\end{proof}

\subsubsection{Proof of Proposition \ref{la5}}

\proof[Proof of Proposition \ref{la5}] The proof of \eqref{ae53} is identical to the proof of
\cite[Eq. (2.6)]{PPTaylor} as it is only based on the limits \eqref{ae51b} and
\eqref{ae51}. The same goes for \eqref{ae62}, which is a corollary of \eqref{ae53}-\eqref{ae71} and whose proof coincides with the proof of \cite[Eq. (2.9)]{PPTaylor}.

Therefore, we only need to prove \eqref{ae71}. Set $f\in C_{0,B}^{2,\alpha}\left(]0,T_0[\times D\right)$ whose support is contained in the interior of
$]0,T_0[\times H$, with $H\subset D$ a compact subset. 
We have
  $$\frac{\TT_{t,T}f(T,x)-f(t,x)}{T-t}=I_{t,T,1}(x)+I_{t,T,2}(x)$$
where
\begin{align}\label{ae74}
 I_{t,T,1}(x)= \int_{H}p(t,x;T,d\x)\frac{f(T,\x)-f(t,x)}{T-t},\qquad
 I_{t,T,2}(x)=-\frac{f(t,x)}{T-t}\int_{\R^d\setminus H}p(t,x;T,d\x),
\end{align}
First note that, by \eqref{ae51b}, if $x\not\in H$ it holds that
\begin{equation}
I_{t,T,1}(x) + I_{t,T,2}(x) = I_{t,T,1}(x) \longrightarrow 0 \quad \text{as }T-t\to 0^+, \qquad \text{unif. w.r.t. } x\in\mathbb{R}^d\setminus H.
\end{equation}
We now consider the case $x\in H$. By \eqref{ae51} the term
$I_{t,T,2}(x)$ is negligible in the limit. As for $I_{t,T,1}(x)$, 
the intrinsic Taylor formula of Corollary \ref{cor_tay_int} yields
\begin{equation}\label{ae63}
\begin{split}
   f(T,\x)-f(t,x)=&(T-t) Y f(t,x)+\sum_{i=1}^{p_0}(\x-e^{(T-t)B}x)_{i}\p_{x_{i}}f(t,x)\\
  &+\frac{1}{2}\sum_{i,j=1}^{p_0}(\x-e^{(T-t)B}x)_{i}(\x-e^{(T-t)B}x)_{j}\p_{x_{i},x_j}f(t,x) + R(t,x;T,\x).
\end{split}
\end{equation}
with $R$ 
such that
\begin{equation}\label{eq_ste101}
\left|R(t,x;T,\x)\right|\le c_{H,B} \Big( |T-t|^{1/2}+\big|\x-e^{(T-t)B}x\big|_B\Big)^{2+\a},
\qquad (t,x),(T,\xi)\in\, ]0,T_0[\times H.
\end{equation}
Next we prove that
\begin{equation}\label{eqste_102}
  \lim_{T-t\to 0^+}\int_{H}
    \big|\x-e^{(T-t)B}x\big|_B^{2+\a}\frac{p(t,x;T,d\x)}{T-t} = 0,\qquad \text{unif. w.r.t. }x\in H.
\end{equation}
For any $x\in H$ and $\d>0$ suitably small we have
\begin{align}
 \int_{H}
  \big|\x-e^{(T-t)B}x\big|_B^{2+\a}\frac{p(t,x;T,d\x)}{T-t} & \leq  C \int_{H\setminus D_{\d}(t,x,T)}
 \frac{p(t,x;T,d\x)}{T-t}\\ &\quad + 
 \d^{\alpha} \int_{D_{\d}(t,x,T)} 
 \big|\x-e^{(T-t)B}x\big|_B^{2} \frac{p(t,x;T,d\x)}{T-t},
\end{align}
where $D_{\d}(t,x,T)=\{\x\in\R^{d}\mid 
 \big|\x-e^{(T-t)B}x\big|_B\le \d\}$ and $C$ is a positive
constant. By \eqref{ae51} 
and \eqref{ae51e} we obtain
\begin{equation}
\limsup_{T-t\to 0^+}\int_{H}
 \big|\x-e^{(T-t)B}x\big|_B^{2+\a}\frac{p(t,x;T,d\x)}{T-t}\leq C_{1}\delta^{\alpha} ,\qquad \text{unif. w.r.t. }x\in H.
\end{equation}
This proves \eqref{eqste_102} since $
\d$ is arbitrary.

Eventually, \eqref{ae71} follows by plugging \eqref{ae63} into \eqref{ae74} and passing to the
limit using \eqref{ae73}, \eqref{ae51c}, \eqref{ae51d} and \eqref{eq_ste101}-\eqref{eqste_102}.
This concludes the proof.
\endproof

\subsubsection{Proof of Theorem \ref{cor:ste101}}
We are now ready to prove the \emph{intrinsic} It\^o formula of Theorem \ref{cor:ste101}.

\begin{proof}[Proof of Theorem \ref{cor:ste101}]
First observe that, integrating \eqref{ae62}, we get the identity
\begin{equation}\label{ae68}
 \left(\TT_{t,T}f(T,\cdot)\right)(x)-f(t,x)=\int_{t}^{T}\TT_{t,\t}\left(
 \Lc f(\t,\cdot)\right)(x)d\t, \qquad T\in\, ]t,T_0[.
\end{equation}
Note that the integrand in \eqref{ae68} is bounded, as a function of $\tau$, because of Assumption \ref{assum1and} and since $f\in C^{2,\alpha}_{0,B}\left(]0,T_{0}[\times D\right)$ and $\TT_{t,\tau}$ is a contraction. {Consider now the process $M^t$ defined through \eqref{ae41}.} For $\t\in[t,T]$ we have
\begin{align}
  E_{t,x}\left[M^{t}_{T}\mid\F^{t}_{\t}\right]&=M^{t}_{\t}+E_{t,x}\left[f(T,X_{T})-f(\t,X_{\t})-\int_{\t}^{T}
  \Lc f(u,X_{u})du\mid\F^{t}_{\t}\right]
=M^{t}_{\t}+{\Phi(\t,X_{\t})}
\end{align}
where, by the Markov property,
\begin{align}
{\Phi(\tau,x)}&=E_{\t,x}\left[f(T,X_{T})-f(\t,x)-\int_{\t}^{T}
\Lc f(u,X_{u})du\right]
\intertext{(by Fubini's theorem)}
  &=\left(\TT_{\t,T}f(T,\cdot)\right)(x)-f(\t,x)-\int_{\t}^{T}\TT_{\t,u}\left(
  \Lc f(u,\cdot)\right)(x)du
\end{align}
which is $0$ by \eqref{ae68}.


To conclude we need to show that
\begin{equation}
Y^t_T := (M^t_T)^2 - \int_t^T  \sum_{i,j=1}^{p_0}    a_{ij}(s,X_s)\partial_{x_i} f(s,X_s)\partial_{x_j} f(s,X_s) \dd s,
\end{equation}
has null $P_{t,x}$-expectation. First note that
\begin{equation}
 \sum_{i,j=1}^{p_0}    a_{ij}(s,X_s)\partial_{x_i} f(s,X_s)\partial_{x_j} f(s,X_s)= \Lc f^2(s,X_{s}) - 2 f(s,X_{s}) \Lc f(s,X_{s}),
\end{equation}
which implies
\begin{equation}
Y^t_T = Y^{1,t}_T +Y^{2,t}_T +Y^{3,t}_T,
\end{equation}
with
\begin{align}
Y^{1,t}_T & =   f^2(T,X_{T})  -\int_{t}^{T}
\Lc f^2(u,X_{u})du + f^2(t,X_{t}) ,\\
Y^{2,t}_T & = 2  f(t,X_{t}) \bigg( \int_{t}^{T}
\Lc f(u,X_{u})du  -  f(T,X_{T})  \bigg)   ,\\
Y^{3,t}_T & = - 2 \int_{t}^T \big(    f(T,X_{T})  -f(u,X_{u})    \big)  
\Lc f(u,X_{u})du  +\bigg( \int_{t}^{T}
\Lc f(u,X_{u})du  \bigg)^2     .
\end{align}
Now, by applying the first part of Theorem \ref{cor:ste101} to $f^2$ and $f$ respectively, it is clear that $Y^{1,t}$ and $Y^{2,t}$ are martingales with null $P_{t,x}$-expectation. Finally, the identity
$$
\bigg( \int_{t}^{T}
\Lc f(u,X_{u})du  \bigg)^2 = 2  \int_{t}^{T} \bigg( \int_u^T 
\Lc f(s,X_{s}) d s \bigg) 
\Lc f(u,X_{u}) d u
$$
along with \eqref{ae41} yields
\begin{align}
Y^{3,t}_T &=  - 2 \int_{t}^T \big(    M^t_T  -M^t_u    \big)  
\Lc f(u,X_{u})du,
\end{align}
which 
shows that $Y^{3,t}$ has null $P_{t,x}$-expectation and thus concludes the proof.

\end{proof}

\section{{Local densities}
}\label{sec:model_proofs}




This section is devoted to the proof of Theorem \ref{th:main}. We will adapt and customize a localization procedure first introduced in \cite{kusuoka-stroock}. 
From now on, throughout the rest of this section, we will assume that the coefficients $a_{ij},a_i$ satisfy Assumption \ref{assum1and} for some $N$, $\alpha$, $M$.

\subsection{Proof of Theorem \ref{th:main}, Part a),b)}\label{sec:proof_main}

We start by observing that the nature of Theorem \ref{th:main} a),b) is strictly local for what
concerns the existence of the transition density and its regularity w.r.t. the forward
space-variable. In other words, it is enough to prove it for $\G(t,x;T,\xi)$ defined for any $0<
t<T<T_0$ and $x\in\R^d$, $\xi \in D'$ and for any sub-domain $D' \Subset D$. 
Therefore, it is not restrictive to assume that there exists a family
$\ta_{ij}  ,\ta_{i}:]0,T_0[\times \R^d\to \R$, $i,j=1,\cdots,p_0$, such that $(\ta_{ij} ,\ta_{i})$
coincide with $(a_{ij}  ,a_{i})$ on $]0,T_0[\times D$ and:
\begin{itemize}
 \item[($\tilde{\text{i}}$)] $\ta_{ij}  ,\ta_{i}  \in C^{N,\alpha}_B(]0,T_0[\times \R^d)$ for any $i,j=1,\dots,p_0$, with all the (Lie) derivatives bounded by $M$;
\item[($\tilde{\text{ii}}$)] the following coercivity condition holds on $\R^d$
\begin{equation}\label{cond:ell-loc_glob}
 M^{-1} |\x|^2 \leq \sum_{i,j=1}^{p_0}
 \ta_{ij}(t,x)\x_{i}\x_{j}\leq M  |\x|^2 ,\qquad t\in\left]0,T_{0}\right[,\ x\in \R^d,\ \x\in\mathbb{R}^{p_0}.
\end{equation}
\end{itemize}
Let us denote by $\tAc_t$ the operator defined as
\begin{align}\label{ae75}
 {\tAc}_{t}:= \frac{1}{2}\sum_{i,j=1}^{p_0}\ta_{ij}(t,x)\p_{x_{i}x_{j}}+\sum_{i=1}^{p_0}\ta_{i}(t,x)\p_{x_{i}} + \langle B x, \nabla_x \rangle  \qquad t\in [0,T_{0}[,\ x\in \R^d,
\end{align}
where $B$ is as in \eqref{eq:B_blocks}. We consider an auxiliary $\R^d$-valued {continuous} strong Markov
process $\tX=(\tX_{t})_{t\in[0,T_0[}$ defined on a space
$\big(\tO,\tF,\big(\tF_{T}^{t}\big)_{0\le t\le T< T_{0}},\big(\tP_{t,x}\big)_{0\le t<
T_{0},x\in\R^d}\big),$
with transition probability function $\tp=\tp(t,x;T,d\x)$, such that $\tX$ is an
\emph{$\tAc_t$-global diffusion on $\R^d$} in the sense of Definition \ref{def:local_diffusion}; in
Appendix \ref{sec:app_markov} we briefly recall the standard construction of such $\tX$.
In particular, it will result in $\tp(t,x;T,d\x)$ having a density $\tG(t,x;T,\xi)$, which coincides with the fundamental solution of the operator
\begin{equation}\label{eq: eq:operator_L_tide}
\tilde{\Lc} = \tAc_t + \partial_t = \frac{1}{2}\sum_{i,j=1}^{p_0}\ta_{ij}(t,x)\p_{x_{i}x_{j}}+\sum_{i=1}^{p_0}\ta_{i}(t,x)\p_{x_{i}} + Y,\qquad t\in [0,T_{0}[,\ x\in \mathbb{R}^d.
\end{equation}
In Lemma \ref{lem:sol_fond_tilde} in Appendix \ref{app:pdes}, some previous results are reported
about the existence of $\tG$, its regularity and some sharp Gaussian upper bounds for $\tG$ and its
derivatives.


\begin{notation}\label{not:cylinders}
For any $x\in\R^d$, $t<T$ and $\eps \in\, ]0,1[$, we set the cylinder
 $$H_{\eps}(t,x;T) :=\, ]t,T[\times S_{\eps}(x),\qquad S_{\eps}(x) := B_1(x-\eps e_1) \cap B_1(x+\eps e_1),$$
where $B_r(x)\subset \R^d$ is the Euclidean (open) ball with radius $r$ centered at $x$ and
$e_1=(1,0,\cdots,0)$ is the first vector of the canonical basis of $\R^d$. We define
\emph{the lateral boundary} and \emph{parabolic boundary} of $H_{\eps}(t,x;T)$, respectively, as
 $$
 \partial_{\Sigma} H_{\eps}(t,x;T) :=\, ]t,T[ \times\partial S_{\eps} (t,x),\qquad
  \partial_P H_{\eps}(t,x;T) :=  \partial_{\Sigma} H_{\eps}(t,x;T)\cup \big(\{ T \} \times \overline{S_{\eps}(x)}\big).
 $$
We also denote by $G=G(t,x;T,\xi)$ the Green function of $(\partial_t+\Ac_t)$ for
$H_{\eps}(0,x_0;T_0)$, which is defined for any $0< t<T<T_0$ and
$x,\xi\in\overline{S_{\eps}(x_0)}$ and enjoys the properties listed in Lemma \ref{lem:green}. In
the latter, we report some preliminary existence (and uniqueness) and regularity results for the
solution of the Cauchy-Dirichlet problem on $H_{\eps}(0,x_0;T_0)$ and in particular for the Green
function $G$.
\end{notation}

Roughly speaking, the following result shows that, prior to the exit time from $S_{\eps}(x_0)$,
$X$ and $\tilde{X}$ have the same law whose density coincides with the Green function $G$. The
proof is based on the crucial fact that the It\^o formula \eqref{ae41} is valid for functions that
are $C^{2}$ in the intrinsic sense, i.e. $C^{2}_{B}$, and not only for functions the are $C^{2}$ in
the Euclidean sense.
\begin{lemma}\label{lemm:ste3}
Let $x_0\in D$ and $\eps \in\, ]0,1[$ such that $\overline{S_{\eps}(x_0)}\subset D$. Then, for any
$(t,x)\in H_{\e}(0,x_{0};T_{0})$ we have
\begin{equation}\label{eq:ste110}
P_{t,x}\left(X_{T}\in A ,\tau^{(t)} > T \right) = \int_{A} G(t,x;T,\xi) \dd \xi =
\tP_{t,x}\big(\tX_{T}\in A , \ttau^{(t)} > T \big) ,\qquad  T\in\, ]t,T_0[,\quad
A\in\mathcal{B}\big(S_{\eps}(x_0)\big),
\end{equation}
where $
\tau^{(t)}$ and $
\ttau^{(t)}$ are the $t$-stopping times defined, respectively, as
\begin{align}\label{eq:def_tau}
\tau^{(t)} := \inf\{ s\geq t : X_s \notin S_{\eps}(x_0)   \},&&
\ttau^{(t)} := \inf\{ s\geq t : \tilde{X}_s \notin S_{\eps}(x_0)   \}.
\end{align}
\end{lemma}
Before to prove Lemma \ref{lemm:ste3} we want to stress the following
\begin{remark}\label{rem:green_estim}
By \eqref{eq:ste110} we obtain
\begin{equation}
\int_{A} G(t,x;T,\xi) \dd \xi \leq \tP_{t,x}\big(\tX_{T}\in A  \big) = \int_{A} \tG(t,x;T,\xi) \dd \xi ,\qquad  
A\in \mathcal{B}\big(S_{\eps}(x_0)\big),
\end{equation}
which implies
\begin{equation}\label{eq:green_estim}
G(t,x;T,\xi)\leq \tG (t,x;T,\xi),\qquad   0< t<T<T_0, \quad x,\xi \in S_{\eps}(x_0).
\end{equation}
\end{remark}
\begin{proof}[Proof of Lemma \ref{lemm:ste3}]
Throughout the proof we will set $\tau:=\tau^{(t)}$ to shorten notation.
Note that \eqref{eq:ste110} is equivalent to
\begin{equation}
E_{t,x}\left[\phi(X_{T}) \caratt_{\tau > T}\right]= \int_{S_{\eps}(x_0)} G(t,x;T,\xi)  \varphi(\xi)  \dd \xi = \tilde{E}_{t,x}\left[\phi(\tX_{T}) \caratt_{\ttau > T}\right] , \qquad 
\phi\in C^{\infty}_0\big( S_{\eps}(x_0) \big).
\end{equation}
Denote by $f$ the unique solution in 
$C^{N+2,\alpha}_B\big(H_{\eps}(0,x_0;T)\big) \,\cap\, C\big((H_{\eps}\cup\partial_{P}H_{\eps})(t,x_0;T)\big)$ 
(see Lemma \ref{lem:green}) of
\begin{equation}\label{ae80}
  \begin{cases}
  \Lc f=0\qquad & \text{on } H_{\eps}(0,x_0;T), \\
    f=0 &  \text{on } \partial_{\Sigma} H_{\eps}(0,x_0;T),\\
    f(T,\cdot) = \varphi &  \text{on } S_{\eps}(x_0) 
  \end{cases}
\end{equation}
which is given by $$ f(t,x) = \int_{S_{\eps}(x_0)} G(t,x;T,\xi)  \varphi(\xi)  \dd \xi. $$ Now, by
Corollary \ref{cor:ste101} combined with Optional Sampling Theorem, the process $M^{t}_{\cdot
\wedge \tau}$ with $M^{t}$ as in \eqref{ae41} and $f$ as in \eqref{ae80}, is an
$\F^{t}$-martingale under $P_{t,x}$: we notice explicitly that, even if $f$ is not defined on
$[0,T_{0}[\times D$ as required by Theorem \ref{cor:ste101}, a standard extension-truncation argument can
be employed. 
Thus
\begin{equation}
E_{t,x}\left[\phi(X_{T}) \caratt_{T<\tau}\right]= E_{t,x}\Big[ f\big(T\wedge \tau,X_{T\wedge \tau}\big)   \Big]=f(t,x). 
\end{equation}
On the other hand, since $\Ac = \tAc$ on $]0,T_0[\times D$, 
$\tX$ is also an \emph{$\locdiff$-local diffusion on 
$D$} and thus 
it holds
\begin{equation}
\tilde{E}_{t,x}\left[\phi(\tX_{T}) \caratt_{T<\ttau}\right]= \tilde{E}_{t,x}\Big[ f\big(T\wedge
\ttau,X_{T\wedge \ttau}\big)   \Big]=f(t,x),\end{equation} which proves \eqref{eq:ste110} and
concludes the proof.



\end{proof}
\begin{proof}[Proof of Theorem \ref{th:main} a),b)] Fix $(t,x)\in\, ]0,T_0[\times \R^d$ and
$x_0\in D$ and let $\eps\in\, ]0,1[$ such that $\overline{S_{\eps}(x_0)}\subseteq D$. 
Let also $U$ and $V$ be two non-empty open subsets such that $x_0\in
U\Subset{V}\Subset{S_{\eps}(x_0)}$.

Define now $\tau^{(t)}_0\equiv t$ and the families $\big(\sigma^{(t)}_n\big)_{n\in\N}$ and
$\big(\tau^{(t)}_n\big)_{n\in\N}$ through the following recursion:
 $$
 \sigma^{(t)}_n:= \inf \big\{  s\geq \tau^{(t)}_{n-1} : X_s\in \overline{V} \big\},
 $$
and $\tau^{(t)}_n: = \tau^{(\s^{(t)}_n)}$ according to notation \eqref{eq:def_tau}, which is
  $$\tau^{(t)}_n= \inf \{  s\geq \sigma^{(t)}_n : X_s\notin S_{\eps}(x_0) \}.$$
Hereafter, whenever it is clear from the context, we will drop the suffix $(t)$ in $\tau^{(t)}_n,\s^{(t)}_n$ to ease the notation.

\vspace{5pt} \noindent \underline{Part a):} note that for any $T\in\, ]t,T_0[$ the event $X_T\in
U$ is included in the disjoint union $\cup_{n\in\N} \big(\sigma^{(t)}_n <T < \tau^{(t)}_n \big)$.
Therefore, for any $A\in\mathcal{B}(U)$ one obtains
\begin{align}
p(t,x;T,A)& = \sum_{n=1}^{\infty} P_{t,x}\big(  X_T\in A,  \sigma_n <T < \tau_n \big) =
\sum_{n=1}^{\infty} P_{t,x}\big(  X_T\in A,  \sigma_n <T < \tau^{(\s_n)} \big) \\ & =
\sum_{n=1}^{\infty} E_{t,x}\Big[ P_{t,x}\big( X_T\in A,  \sigma_n <T < \tau^{(\s_n)} \big|
\F_{\s_n} \big)\Big] \\ &=   \sum_{n=1}^{\infty} E_{t,x}\Big[  P_{t,x}\big( X_T\in A,  T <
\tau^{(\s_n)} \big| \F_{\s_n} \big)  \caratt_{\sigma_n <T} \Big] \intertext{(by the strong Markov
property)} 
& = \sum_{n=1}^{\infty} E_{t,x}\Big[  
P_{s,y}\big(X_{T}\in A ,\tau^{(s)} > T \big)\big|_{s=\sigma_n,y=X_{\sigma_n}} \caratt_{ \sigma_n
<T } \Big]
= \sum_{n=1}^{\infty} E_{t,x} \bigg[  \int_{A} \tGi(\s_n,X_{\s_n};T,\xi) \dd \xi\,    \caratt_{
\sigma_n <T }  \bigg] ,\label{eq:ste220}
\end{align}
where we used \eqref{eq:ste110} in the last equality. Our derivation of \eqref{eq:ste220} follows
closely the original argument by \cite{kusuoka-stroock} even if here we go a step further using
the representation in terms of the Green kernel \eqref{eq:ste110}, which is crucial in the
subsequent study of the regularity properties of the local density of $X$.

From \eqref{eq:ste220} and since $x_{0}$ is arbitrary, it follows that $p(t,x;T,\cdot)$ is
absolutely continuous w.r.t. the Lebesgue measure on $D$ and therefore admits a density
$\G(t,x;T,\xi)$. Moreover, for any $x_0\in D$ and $\eps\in\, ]0,1[$ such that
$\overline{S_{\eps}(x_0)}\subset D$, we have the local representation
\begin{equation}\label{eq:repres_Gamma}
 \G(t,x;T,\xi) =  \sum_{n=1}^{\infty}  E_{t,x}\big[  
 \tGi(\s_n,X_{\s_n};T,\xi) \caratt_{ \sigma_n <T } 
 \big], \qquad T\in\, ]t,T_{0}[,\ \xi\in S_{\e}(x_{0}).
\end{equation}
Assume now that there exists $C_1>0$, independent of $t$, $x$ and $T$, such that
\begin{equation}\label{eq:stima_somma}
\sum_{n=1}^{\infty} P_{t,x} \big(\sigma^{(t)}_n <T\big) \leq C_1, \qquad T\in\, ]t,T_0[.
\end{equation}
Then, by the continuity of $G(t,x;T,\cdot)$ on $S_{\e}(x_{0})$ combined with
\eqref{eq:green_estim} and the estimates \eqref{eq:bound_gamm_tild_forward}, it follows that
$\G(t,x;T,\cdot)$ is continuous and bounded on $S_{\e}(x_{0})$, uniformly w.r.t. $x\in\R^d$.
Therefore, to conclude the proof of Part a) we only need to prove \eqref{eq:stima_somma}.
%

Start by observing that
\begin{equation}\label{eq:ste204}
\sum_{n=2}^{\infty} P_{t,x}\big( \sigma_n <T  \big) \leq \sum_{n=1}^{\infty} P_{t,x}\big( \tau_n <T \big),
\end{equation}
and that, by classical maximal estimates (e.g. \cite{pascuccibook}, p. 296), it holds that
\begin{equation}\label{eq:maximal_estimates}
P_{s,y}\big( \tau^{(s)} < T  \big) \leq C e^{-\frac{1}{C(T-s)}} , \qquad   0<s<T<T_0,\quad  y\in \partial V,
\end{equation}
where $C>0$ only depends on $T_0$ and $\Ac_t$, but not on $s,T$ and $y$.
Therefore, for any $n\geq 1$ we have
\begin{align}
P_{t,x}\big( \tau_n < T  \big) &= E_{t,x}\big[ E_{{t,x}}[  \caratt_{\tau^{(\s_n)} <T}  | \F_{\s_n}
] \big] = E_{{t,x}}\big[ E_{{s,y}}[  \caratt_{\tau^{(s)} <T}  ]|_{s=\s_n,y=X_{\s_n}} \big]
\intertext{(by \eqref{eq:maximal_estimates})} &\leq C e^{-\frac{1}{C(T-t)}} P_{t,x}\big( \sigma_n
< T  \big)\leq C e^{-\frac{1}{C(T-t)}} P_{t,x}\big( \tau_{n-1} < T  \big),
\end{align}
which 
yields $P_{t,x}\big( \tau_n < T  \big)\leq \big(C e^{-\frac{1}{C(T-t)}}\big)^n$. 
This combined with \eqref{eq:ste204} proves \eqref{eq:stima_somma} for $T-t\leq T^*$ and for a positive $T^*$ suitably small only dependent on $T_0$ and $\Ac_t$. To prove \eqref{eq:stima_somma} for a generic $T\in\, ]t,T_0[$ consider a partition $t=t_0<t_1<\cdots<t_N = T,$ such that $t_{k+1}-t_k< T^*$. 
Define $i_k:=\inf\{ n\in \N : \s^{(t)}_n \geq t_k \}$. We first observe that
\begin{align}
\sum_{n=1}^{\infty}   \caratt_{t_k\leq \sigma^{(t)}_n <t_{k+1}}& =\sum_{n=i_k}^{\infty}   \caratt_{ \sigma^{(t)}_n <t_{k+1}} = \sum_{m=0}^{\infty}   \caratt_{\sigma^{(t)}_{i_k+m} <t_{k+1}}
\intertext{(since we have  $\s^{(t)}_{i_k} \in \{ \s^{(t_k)}_{1}, \s^{(t_k)}_{2} \}$ and thus, by induction, $\s^{(t)}_{i_k+m}\geq \s^{(t_k)}_{1+m}$)}
&\leq \sum_{m=1}^{\infty}   \caratt_{ \sigma^{(t_k)}_m <t_{k+1}}, \qquad k=0,\cdots,N-1.
\end{align}
Hence
\begin{align}
\sum_{n=1}^{\infty} P_{t,x} \big(\sigma^{(t)}_n <T\big)
  &=\sum_{k=0}^{N-1} \sum_{n=1}^{\infty}  P_{t,x} \big(t_k\leq \sigma^{(t)}_n <t_{k+1}\big)  \leq \sum_{k=0}^{N-1} \sum_{m=1}^{\infty}  P_{t,x} \big(\sigma^{(t_k)}_m <t_{k+1}\big)\\
  & = \sum_{k=0}^{N-1} \sum_{m=1}^{\infty} E_{t,x}\big[ P_{t,x} \big(\sigma^{(t_k)}_m <t_{k+1}\big| \F^t_{t_k} \big) \big] =  \sum_{k=0}^{N-1} E_{t,x}\bigg[ \sum_{m=1}^{\infty} P_{t_k,y} \big(\sigma^{(t_k)}_n <t_{k+1} \big)\big|_{y=X_{t_k}}  \bigg],
  \end{align}
  which proves \eqref{eq:stima_somma}.

\vspace{5pt} \noindent \underline{Part b):}
we 
only prove the statement for $N=2$, the general case being analogous. 
By combining the representation of $G$ in \cite{Francesco}, p. 36, with the internal estimates in the same reference that are reported in Theorem \ref{th:schauder_est} below, it follows that
\begin{equation}\label{eq:bound_deriv_green}
|\partial_{\xi_{i}} \tGi(s,y;T,\xi)| + |\partial_{\xi_{i}\xi_{j}} \tGi(s,y;T,\xi)| + |Y_{T,\xi} \tGi(s,y;T,\xi)| \leq C_2,  \qquad 0< s < T<T_0,\quad y\in \partial V, \quad \xi\in U,
\end{equation}
for any $i,j=1,\cdots,p_0$. 
This and \eqref{eq:stima_somma} allow us to employ bounded convergence theorem and differentiate twice under the sign of expectation 
the right-hand side of \eqref{eq:repres_Gamma} w.r.t. $\xi$. For any $T\in\, ]t,T_0[$, $\xi\in U$ we obtain:
\begin{align}
\partial_{\xi_i}\G(t,x;T,\xi) &=  \sum_{n=1}^{\infty}  E_{t,x}\big[  
\partial_{\xi_i} \tGi(\s_n,X_{\s_n};T,\xi) \caratt_{ \sigma_n <T } 
 \big],
 \\ \partial_{\xi_i\xi_j}\G(t,x;T,\xi) &=  \sum_{n=1}^{\infty}  E_{t,x}\big[  
\partial_{\xi_i\xi_j} \tGi(\s_n,X_{\s_n};T,\xi) \caratt_{ \sigma_n <T } 
 \big],
\end{align}
for 
$i,j=1,\cdots,p_0$. As for $Y \G(t,x;\cdot,\cdot)$, 
we have
\begin{align}
Y_{T,\xi} \G(t,x;T,\xi) &= \lim_{h\to 0} \frac{\G\big(t,x;T+h,e^{hB}\xi\big) - \G(t,x;T,\xi)}{h}  \\
& = \lim_{h\to 0} \frac{1}{h} \sum_{n=1}^{\infty} \Big( E_{t,x}\big[  
\tGi(\s_n,X_{\s_n};T+h,e^{hB}\xi) \caratt_{ \sigma_n <T+h } 
 \big] -
 E_{t,x}\big[  
\tGi(\s_n,X_{\s_n};T,\xi) \caratt_{ \sigma_n <T } 
 \big] \Big) \\
& = \lim_{h\to 0} \frac{1}{h} \sum_{n=1}^{\infty}  E_{t,x}\big[  
\underbrace{\tGi(\s_n,X_{\s_n};T+h,e^{hB}\xi) \caratt_{T\leq \sigma_n <T+h }}_{=:I_{1,n}(h)} 
 \big] \\
 & \quad + \lim_{h\to 0} \frac{1}{h} \sum_{n=1}^{\infty}  E_{t,x}\big[  
\underbrace{\big(  \tGi(\s_n,X_{\s_n};T+h,e^{hB}\xi) - \tGi(\s_n,X_{\s_n};T,\xi) \big)
\caratt_{ \sigma_n <T }}_{I_{2,n}(h)} 
 \big].
\end{align}
Remark \ref{rem:green_estim} together with 
\eqref{eq:bound_gamm_tild_forward} on one hand, and mean value theorem on the other, yield
\begin{align}
\frac{1}{h}I_{1,n}(h) &\leq C \caratt_{T\leq \sigma_n <T+h }, \\
\frac{1}{h}I_{2,n}(h) &= Y_{T,\xi} \tGi\big(\s_n,X_{\s_n};T+\tilde{h},e^{\tilde h B}\xi\big) \caratt_{ \sigma_n <T }, \quad \text{with }   |\tilde h|\leq h.
\end{align}
Here the low index in $Y_{T,\xi}$ is meant to stress that $Y$ is computed with respect to the variables $(T,\xi)$.
By \eqref{eq:stima_somma} and \eqref{eq:bound_deriv_green} we can apply bounded convergence theorem and obtain
\begin{equation}
Y_{T,\xi}\G(t,x;T,\xi) =  \sum_{n=1}^{\infty}  E_{t,x}\big[  \caratt_{ \sigma_n <T } 
Y_{T,\xi} \tGi(\s_n,X_{\s_n};T,\xi) 
 \big], \qquad T\in\, ]t,T_0[,\quad \xi\in U.
\end{equation}
Proceeding analogously, by employing again the Schauder estimates 
reported in Theorem \ref{th:schauder_est}, one also proves
\begin{equation}
\partial_{\x_i} \G(t,x;\cdot,\cdot)\in C^{1+\alpha}_Y(]t,T_0[ \times U), \quad  \partial_{\x_i\xi_j}\G(t,x;\cdot,\cdot),Y\G(t,x;\cdot,\cdot) \in C^{0,\alpha}_B(]t,T_0[ \times U), \qquad i,j=1,\cdots,p_0,
\end{equation}
which is $ \G(t,x;\cdot,\cdot)\in C^{2,\alpha}_B(]t,T_0[ \times U)$.
Eventually, $ \G(t,x;\cdot,\cdot)\in C^{2,\alpha}_B(]t,T_0[ \times D)$ follows from the fact that $x_0$ is arbitrary. The fact that $\G(t,x;\cdot,\cdot)$ solves \eqref{eq:forward_kolm} is now a straightforward consequence of Remark \ref{rem:kolm_forw_dist}, by integrating by parts the left-hand side of \eqref{ae42_bis} and since $f$ is arbitrary.


\end{proof}

\subsection{Proof of Theorem \ref{th:main}, Part c)}

Hereafter throughout this section we asssume 
Assumption \ref{assum:feller} to be in force as well.

\begin{notation}
For any $T\in\, ]0,T_{0}[$ and any $\phi$ bounded and Borel-measurable on $\R^d$ (in short
$\phi\in m\mathcal{B}_b$), let $u_{\varphi,T}:]0,T[\times D \to \R$ be the function defined as
$u_{\varphi,T}(t,x) :=  \big(\TT_{t,T} \varphi \big)(x)$.
\end{notation}

\begin{lemma}\label{lem:expect_value_pde}
Let 
$T\in\, ]0,T_0[$ and $\varphi\in m\mathcal{B}_b$ such that $u_{\varphi,T} \in C( ]0,T[\times D )$.
Then, $u_{\varphi,T}
\in C^{N+2,\alpha}_B( ]0,T[\times D )$ 
and solves the backward Kolmogorov equation \eqref{eq:backward_kolm}.
\end{lemma}
Again, the proof is based on the crucial fact that the It\^o formula \eqref{ae41} is valid for
functions that are $C^{2}$ in the intrinsic sense, i.e. $C^{2}_{B}$, and not only for functions the
are $C^{2}$ in the Euclidean sense.
\begin{proof}
Let $\delta>0$, $x_0\in D$, $\eps \in\, ]0,1[$ such that $\overline{S_{\eps}(x_0)}\subset D$ and
denote by $f$ the unique solution in $C^{N+2,\alpha}_B\big(H_{\eps}(0,x_0;T-\delta)\big) \,\cap\,
C\big((H_{\eps}\cup\partial_{P}H_{\eps})(0,x_0;T-\delta)\big)$ (see Lemma \ref{lem:green}) of
\begin{equation}\label{ae80_ter_bis}
  \begin{cases}
 \Lc f=0\qquad & \text{on } H_{\eps}(0,x_0;T-\delta), \\
    f=u_{\phi,T} &  \text{on } \partial_{P} H_{\eps}(0,x_0;T-\delta).
  \end{cases}
\end{equation}
For any $t\in\, ]0,T-\delta[$ let now $\tau=\tau^{(t)}$ be the $t$-stopping time as defined in
\eqref{eq:def_tau}. By Theorem \ref{cor:ste101} combined with Optional Sampling Theorem, the
process $M^{t}_{\cdot \wedge \tau}$, with $M^{t}$ as in \eqref{ae41} and $f$ as in
\eqref{ae80_ter_bis}, is an $\F^{t}$-martingale under $P_{t,x}$. Thus
\begin{align}
f(t,x)&= E_{t,x}\big[ u_{\phi,T}\big((T-\d)\wedge \tau,X_{(T-\d)\wedge \tau}\big)   \big]
= E_{t,x}\big[ E_{s,y}[\phi(X_T) ]|_{s=(T-\d)\wedge \tau , y=X_{(T-\d)\wedge \tau} }  \big]
\intertext{(by Strong Markov property)}
&= E_{t,x}\big[   E_{t,x}[  \phi(X_T) | \F_{(T-\d)\wedge \tau} ]   \big] = E_{t,x}[  \phi(X_T) ] = u_{\phi,T}(t,x).
\end{align}
Since $x_0$ and $\d$ are arbitrary, then  $u_{\varphi,T}
\in C^{N+2,\alpha}_B( ]0,T[\times D )$ 
and solves 
\eqref{eq:backward_kolm}.
\end{proof}

\begin{lemma}[\bf Strong Feller property]\label{lemm:strong_feller}
For any $T\in\, ]0,T_{0}[$ and any $\varphi\in m\mathcal{B}_b$, $u_{\varphi,T} \in C( ]0,T[\times
D )$.
\end{lemma}

\begin{proof}
First note that, 
by Assumption \ref{assum:feller} combined with Lemma \ref{lem:expect_value_pde}, $u_{\psi,T} \in C^{N+2,\alpha}_B( ]0,T[\times D )$ and solves the backward Kolmogorov equation \eqref{eq:backward_kolm}, for any $\psi\in C_b(\R^d)$. 
Thus, by the internal Schauder estimates 
reported in Theorem \ref{th:schauder_est},
for any 
bounded domain 
${V}\Subset\, ]0,T[\times D$  we have
\begin{equation}
 \| u_{\psi,T} \|_{C^{0,\alpha}_B(V)}  \leq C \sup_{]0,T[\times D}  |u_{\psi,T}| \leq C 
 \| \psi \|_{\infty},
\end{equation}
where $C$ is a positive constant independent of $\psi$. In particular, by Theorem 
\ref{th:main_tay}, for any $(t_0,x_0)\in\, ]0,T[\times D$ there exists a neighborhood $U_{(t_0,x_0)}$ 
that
\begin{equation}
\big| u_{\psi,T}(t,x) - u_{\psi,T}(t',x') \big| \leq C \| \phi \|_{\infty} \| (t',x')^{-1} \circ (t,x) \|_B , \qquad (t,x),(t',x')\in U_{(t_0,x_0)}, \qquad
\end{equation}
for any $\psi\in C_b(\R^d)$ such that $\| \psi \|_{\infty}\leq \| \phi \|_{\infty}$. 
Therefore, in order to prove that $u_{\phi,T}$ is continuous in $(t_0,x_0)$ it suffices to prove that, for any 
$(t,x)
\in U_{(t_0,x_0)}$, there exist a sequence 
of functions $\psi_n\in C_b(\R^d)$ with $\| \psi_n \|_{\infty}\leq \| \phi \|_{\infty}$ such that
\begin{equation}\label{eq:convergence_psi}
u_{\psi_n,T}(t,x) \longrightarrow u_{\phi,T}(t,x), \quad u_{\psi_n,T}(t_0,x_0) \longrightarrow u_{\phi,T}(t_0,x_0) \qquad \text{as } n\to \infty.
\end{equation}
To see this, let $\mu$ be the measure on $\mathcal{B}(\R^d)$ defined as
\begin{equation}
\mu(\dd z) = p(t,x;T,\dd z) + p(t_0,x_0;T,\dd z).
\end{equation}
Note that we have
\begin{equation}\label{eq:measure_mu_ord}
 p(t,x;T,\dd z),\, p(t_0,x_0;T,\dd z) \ll \mu (\dd z).
\end{equation}
Moreover, $\mu$ is a finite measure and thus, by Proposition 3.16 in \cite{AmbrosioDaPrato}, there exists a sequence of $\psi_n\in C_b(\R^d)$ with $\| \psi_n \|_{\infty}\leq \| \phi \|_{\infty}$ such that
\begin{equation}
\| \psi_n -\phi \|_{L^1(\mathcal{B}(\R^d),\mu)} \longrightarrow 0  \quad \text{as }n\to \infty.
\end{equation}
Therefore, by \eqref{eq:measure_mu_ord}, $\psi_n\to \phi$ both $p(t,x;T,\dd z)$- and
$p(t_0,x_0;T,\dd z)$-almost everywhere. Thus bounded convergence theorem yields
\eqref{eq:convergence_psi} and concludes the proof.
\end{proof}

\begin{proposition}\label{cor:back_kolm}
For any $T\in\, ]0,T_{0}[$ and any $\phi\in m\mathcal{B}_b$, $u_{\varphi,T}
\in C^{N+2,\alpha}_B( ]0,T[\times D )$ 
and solves the backward Kolmogorov equation \eqref{eq:backward_kolm}.
\end{proposition}
\begin{proof}
It is an immediate consequence of Lemmas \ref{lem:expect_value_pde} and \ref{lemm:strong_feller}.
\end{proof}

\begin{lemma}\label{lem:chap_kolm_density}
For any $0< t <T<T_0$ and $\xi\in D$,
the function $\G(t,\cdot;T,\xi)\in m\mathcal{B}_b$, and 
\begin{equation}\label{eq:chap_kolm_density}
\G(t,x;T,\xi) =  E_{t,x}[\G(s,X_s;T,\xi)] = \big(\TT_{t,s} \G(s,\cdot;T,\xi) \big)(x) 
,\qquad 
t<s<T,\quad 
x\in\R^d
.
\end{equation}
\end{lemma}

\begin{proof}
The boundedness is already contained in Part a) of Theorem \ref{th:main}.
We first 
prove the measurability of $\G(t,\cdot;T,\xi)$.
Let $(\phi_n)_{n\in\N}$ a family of functions in $C_0(D)$ such that $\phi_n\to \d_{\xi}$, i.e.
\begin{equation}
\int_{D} f(y) \phi_n(y) \dd y \to f(\xi)\quad \text{as }n\to\infty , \qquad f\in C(D).
\end{equation}
Therefore, since $\Gamma(t,x;T,\cdot)\in C(D)$ (again by Part a) of Theorem \ref{th:main}), we have 
\begin{equation}\label{eq:lim_density}
\Gamma(t,x;T,\xi) = \lim_{n\to\infty}  \int_{D} \phi_n(y) \Gamma(t,x;T,y)  \dd y = \lim_{n\to\infty} u_{\phi_n,T}(t,x), \qquad x\in \R^d,
\end{equation}
%
and since $u_{\phi_n,T}(t,\cdot)$ is continuous (by Assumption \ref{assum:feller}), 
$\G(t,\cdot;T,\xi)$ is measurable as it is the pointwise limit of a sequence of measurable functions.

%

We now prove \eqref{eq:chap_kolm_density}. 
By \eqref{eq:lim_density} along with Markov property, it holds that  
\begin{align}
\Gamma(t,x;T,\xi) & =\lim_{n\to\infty} E_{t,x}[\phi_n(X_T)] =\lim_{n\to\infty} E_{t,x}\big[ E_{t,x}[\phi_n(X_T)|\F^t_s] \big] =\lim_{n\to\infty} E_{t,x}\big[ E_{s,y}[\phi_n(X_T)]|_{y=X_s} \big]  \\
& = \lim_{n\to\infty} E_{t,x}\big[ u_{\phi_n,T}(s,X_s) \big]=  E_{t,x}[\G(s,X_s;T,\xi)] , 
\end{align}
where, in the last equality, we employed again \eqref{eq:lim_density} with $t=s$ and $x=X_s$ along with bounded convergence theorem (it is not restrictive to assume $\|\phi_n\|_{L^1(D)} = 1$ and thus, since $\G(s,x;T,\cdot)$ is locally bounded on $D$ uniformly w.r.t. $x\in\R^d$, $u_{\phi_n,T}(s,X_s)$ bounded uniformly w.r.t. $n$). 
\end{proof}

\begin{proof}[Proof of Theorem \ref{th:main},c)] It is a straightforward consequence of Proposition \ref{cor:back_kolm} and Lemma \ref{lem:chap_kolm_density}.
\end{proof}

\appendix
\section{Preliminary PDE results}\label{app:pdes}

In this appendix we collect some useful results about the operators $\Lc$ in \eqref{eq:operator_L} and $\tilde\Lc$ 
in \eqref{eq: eq:operator_L_tide}, under the Assumptions \ref{ass:hypo} and \ref{assum1and} and with the coefficients $\ta_{ij}  ,\ta_{i}:]0,T_0[\times \R^d\to \R$, $i,j=1,\cdots,p_0$, satisfying the conditions ($\tilde{\text{i}}$)-($\tilde{\text{ii}}$) at the beginning of Section \ref{sec:proof_main}. 

\begin{theorem}\label{lem:sol_fond_tilde}
There exists a unique fundamental solution 
for $
\Lc$, namely a {continuous}
non-negative function $\tG=\tG(t,x;T,\xi)$ defined for any $0< t<T<T_0$ and $x,\xi\in
\R^d$ enjoying the following properties:
\begin{itemize}
\item[a)] for any $(T,\xi)\in\, ]0,T_0[\times \R^d$, the function $\tG(\cdot,\cdot;T, \xi)\in
C^{N+2,\alpha}_B(]0,T[\times \R^d)$ and is a solution of
\begin{equation}\label{eq:fund_sol_cauchy}
  \begin{cases}
  \tilde\Lc u=0\qquad & \text{on } ]0,T[\times \R^d, \\
     u(T,\cdot) = \delta_{\xi},&
  \end{cases}
\end{equation}
where the terminal condition is in the distributional sense, i.e.
  $$\lim_{
  t\to T^-}\int_{\R^{d}}\tG(t,x;T,\xi)\phi(x)dx=\phi(\xi),\qquad \phi\in C_{0}(\R^{d});$$

\item[b)] if $N\geq 2$, then for any $(t,x)\in\, ]0,T_0[\times \R^d$, the function $\tG(t,x;\cdot,\cdot)\in C^{N,\alpha}_B(]t,T_0[\times \R^d)$ and is  solution to
\begin{equation}\label{eq:fund_sol_cauchy_adj}
  \begin{cases}
  \tilde\Lc^* u=0\qquad & \text{on } ]t,T_0[\times \R^d, \\
     u(t,\cdot) = \delta_{x}, &
  \end{cases}
\end{equation}
where $
\tilde\Lc^*$ is the formal adjoint of $
\tilde\Lc$.

%

\item[c)] for any $\alpha,\beta\in\N_0^{d}$ and $p,q\in\N_0$ with $2p+[\alpha]_B\leq N+2$ and $2q + [\beta]_B\leq 
N \caratt_{[2,\infty[}(N)$, we have
\begin{align}
\big|Y^p_{t,x} \partial_x^{\alpha}  \tG(t,x;T,\xi) \big| & \leq C (T-t)^{-\frac{
2p+[\alpha]_B}{2}} \bar{\Gamma}_M (t,x;T,\xi),\\
\big| Y^{q}_{T,\xi} \partial_{\xi}^{\beta}  \tG(t,x;T,\xi) \big| & \leq C (T-t)^{-\frac{
2q+[\beta]_B}{2}} \bar{\Gamma}_M (t,x;T,\xi), \label{eq:bound_gamm_tild_forward}
\end{align}
for any $0< t<T<T_0$ and $x,\xi\in\R^d$, where $C$ is a positive constant that only depends on
$B,M,T_0$ and $N$ and where
\begin{equation}
 \bar{\Gamma}_M (t,x;T,\xi) = \frac{1}{\sqrt{(2\pi)^d \det\Cv(T-t)}}  \exp\left(-\frac{1}{2}\big\langle\Cv^{-1}(T-t) \big(y - e^{(T-t)B}x\big),
    \big(y -e^{(T-t)B}x\big)\big\rangle\right),
\end{equation}
with
\begin{align}\label{eq:covariance_mean}
 \Cv(s)= \int_0^s e^{r B}\begin{pmatrix}
    M I_{p_0} & 0_{p_0\times (d-p_0)} \\
    0_{(d-p_0)\times p_0} & 0_{(d-p_0)\times (d-p_0)}
  \end{pmatrix}e^{r B^*} d r.
\end{align}
\end{itemize}
\end{theorem}
\begin{proof}
See Theorem 1.4 in \cite{DiFrancescoPascucci2}.
\end{proof}

In the next result, the sets $S_{\eps}(x_0)\subseteq \Rd$ are as defined in Notation
\ref{not:cylinders}.
\begin{theorem}\label{lem:green}
Let $x_0\in D$ and $\eps \in\, ]0,1[$ such that $\overline{S_{\eps}(x_0)}\subseteq D$. Then, for
any $T\in\, ]0,T_0[$ and $h\in C\big(\partial_{P}{H_{\eps}(0,x_0;T)}\big) $, there exists a unique
solution in $ C^{N+2,\alpha}_B\big(H_{\eps}(0,x_0;T)\big) \,\cap\,
C\big((H_{\eps}\cup\partial_{P}H_{\eps})(0,x_0;T)\big)$ to
 \begin{equation}\label{ae80_ter}
  \begin{cases}
  \Lc f=0,\qquad & \text{on } H_{\eps}(0,x_0;T), \\
    f=h, &  \text{on } \partial_{P} H_{\eps}(0,x_0;T).
  \end{cases}
\end{equation}
Moreover, if $h|_{\partial_{\Sigma}{H_{\eps}(0,x_0;T)}}\equiv 0$, then the following representation holds:
\begin{equation}
  f(t,x) = \int_{S_{\eps}(x_0)} G(t,x;T,\xi)  h(T,\xi)  \dd \xi, \qquad (t,x)\in \left(H_{\eps}\cup\partial_{P}H_{\eps}\right)(0,x_0;T),
\end{equation}
where $G$ denotes the Green function of $
\Lc$ for $H_{\eps}(0,x_0;T_0)$, namely a {continuous}
non-negative function $G(t,x;T,\xi)$ defined for any $0< t<T<T_0$ and $x,\xi\in
\overline{S_{\eps}(x_0)}$ enjoying the following properties:
\begin{itemize}
\item[a)] for any $(T,\xi)\in H_{\eps}(0,x_0;T_{0})$,
the function $G(\cdot,\cdot;T,\xi)\in C^{N+2,\alpha}_B\big(H_{\eps}(0,x_0;T)\big) \,\cap\,
C\big((H_{\eps}\cup\partial_{\Sigma}H_{\eps})(0,x_0;T)\big)$ and solves
 \begin{equation}\label{ae80_quat}
  \begin{cases}
  \Lc f=0\qquad & \text{on }
   H_{\eps}(0,x_0;T)
  , \\
    f= 0 &  \text{on } \partial_{\Sigma} H_{\eps}(0,x_0;T), \\
       f(T,\cdot) = \delta_{\xi} &  \text{on } S_{\eps}(x_0); 
  \end{cases}
\end{equation}
\item[b)] if $N\geq 2$, then $G$ is also the Green function of the
formal adjoint $
\Lc^*$ for $H_{\eps}(0,x_0;T_0)$. In particular, for any
$(t,x)\in H_{\e}(0,x_{0};T_{0})$, the function $G(t,x;\cdot,\cdot)\in
 C^{N,\alpha}_B\big({H_{\eps}(t,x_0;T_0)}\big) \,\cap\, C\big((H_{\eps}\cup\partial_{\Sigma}H_{\eps})(t,x_0;T_{0})\big)$ and solves
 \begin{equation}\label{ae80_quatb}
  \begin{cases}
  \tilde\Lc^* f=0\qquad & \text{on } {H_{\eps}(t,x_0;T_0)}, \\
   f= 0 &  \text{on } \partial_{\Sigma} H_{\eps}(t,x_0;T_0), \\
       f(t,\cdot) = \delta_{x} &  \text{on } S_{\eps}(x_0). 
  \end{cases}
\end{equation}
\end{itemize}

\end{theorem}
\begin{proof}
See \cite{Manfredini} and \cite[Section 4]{Francesco}.
\end{proof}

\begin{theorem}\label{th:schauder_est}
Let $Q\subset\Rdd$ be a bounded domain, 
and $u$ be a bounded function in $C^{2+\alpha}_B (Q)$ such that $\tilde\Lc u = 0$ on $Q$. Then, for any domain $Q_0
\Subset Q$, there exists $c>0$ such that
\begin{equation}
\sup_{Q_0}\bigg( \sum_{i=1}^{p_0} |\partial_{x_i} u |  + \sum_{i,j=1}^{p_0} |\partial_{x_i x_j} u | + |Y u| \bigg) + \|{u}\|_{C^{0,\a}_{B}(Q_0)}+\|{u}\|_{C^{1,\a}_{B}(Q_0)}+\|{u}\|_{C^{2,\a}_{B}(Q_0)} \leq c \sup_{Q} u .
\end{equation}
\end{theorem}
\begin{proof}
It is a particular case of \cite[Theorem 1.3]{Francesco}.
\end{proof}

\section{Fundamental solutions and Markov processes}\label{sec:app_markov}
We recall some basic notions about Markov processes as given in \cite{FriedmanSDE1} and
\cite{StroockVaradhan}. A \emph{transition distribution} is a kernel $p(t,x;T,\cdot)$ that
satisfies:\vspace{-5pt}
\begin{itemize}
 \item[1)]  ${p}(t,x;T,\cdot)$ is a probability measure on $(\mathbb{R}^d,\mathcal{B}(\mathbb{R}^d))$ for all $ 0\leq t<T< T_0$ and $x\in \R^d$;\vspace{-5pt}
 \item[2)]  ${p}(t,\cdot;T,A)$ is $\mathcal{B}(\mathbb{R}^d)$-measurable for any $ 0\leq t<T< T_0$ and $A\in\mathcal{B}(\mathbb{R}^d)$;\vspace{-5pt}
 \item[3)] if $0\leq t<s<T< T_0$, $x\in\mathbb{R}^d$ and $A\in\mathcal{B}(\mathbb{R}^d)$, the following Chapman-Kolmogorov identity holds:
\begin{equation}\label{CK}
{p}(t,x;T,A)=\int_{\mathbb{R}^d}{p}(s,\xi;T,A){p}(t,x;s,d\xi).
\end{equation}
\end{itemize}
A \emph{Markov process with transition distribution $p$} is a stochastic process $X=(X_{t})_{0\leq
t< T_0}$ defined on the quartet $\left(\O,\F,(\F_{T}^{t})_{0\le t\le T< T_{0}},(P_{t,x})_{0\le t<
T_{0},x\in\R^d}\right)$ such that:\vspace{-5pt}
\begin{itemize}
\item[(a)] $(\O,\F)$ is a measurable space and $(\F_{T}^{t})_{0\le t\le T< T_{0}}$ is a family of
filtrations satisfying $\F_{T}^{t} \subseteq \F_{T'}^{t'}$ for $t'\le t,T\le T'$ and
$\F=\F^0_{T_0}$ (i.e. $\F$ is the smallest $\sigma$-algebra containing all $\F_{T}^{t}$);
\vspace{-5pt}\item[(b)] $(X_T)_{t\leq T<T_0}$ is adapted to $\F^t$ for any $t\in[0,T_0[$;
\vspace{-5pt}\item[(c)] for any $(t,x)\in[0,T_0[\times \Rd$, $P_{t,x}$ is a probability measure on
$(\O,\F^{t}_{T_0})$ satisfying
\begin{align}
P_{t,x}(X_t =x)&= 1,   \\
P_{t,x}(X_T \in A | \F^t_s)&= p(s,X_s;T,A) ,\qquad t\leq s < T< T_0,\quad A\in\mathcal{B}(\mathbb{R}^d).
\end{align}
\end{itemize}
Theorem 2.2.2 in \cite{StroockVaradhan} guarantees that for any transition distribution $p$ there
exists a Markov process $X=(X_{t})_{t\in[0,T_0[}$ having $p$ as transition distribution.

A transition distribution can be defined from a differential operator 
$\tilde\Lc$ of the
form \eqref{eq: eq:operator_L_tide} satisfying conditions
($\tilde{\text{i}}$)-($\tilde{\text{ii}}$) at the beginning of Section \ref{sec:proof_main}.
Indeed, if $\tilde\Gamma$ denotes the fundamental solution of $\tilde\Lc$ in Theorem \ref{lem:sol_fond_tilde}
then
\begin{equation}\label{ee100}
 \tilde p(t,x;T,A):=\int_A \tilde \Gamma(t,x;T,\xi)d\xi,\qquad 0\le t<T<T_0,\ x\in\mathbb{R}^d, \
 A\in\mathcal{B}(\mathbb{R}^d)
 \end{equation}
defines a transition distribution.
 In virtue of the properties of $\Gamma$ in Theorem \ref{lem:sol_fond_tilde}, if $\tilde p$ is as in
\eqref{ee100} then the associated Markov process admits a continuous version and
is an \emph{$\Ac_t$-global diffusion on $\R^d$} in the sense of Definition \ref{def:local_diffusion}. 
This is a consequence of Kolmogorov continuity theorem (see, for instance, Theorems 2.1.6 and
2.2.4 in \cite{StroockVaradhan}) and the estimate given in the following
\begin{lemma}\label{lla}
Let $\tilde X=(\tilde X_{t})_{t\in[0,T_0[}$ be a Markov process with transition distribution $\tilde p$ in
\eqref{ee100}. Then for any $q\ge 1$ there exists a positive constant $C$ such that
\begin{equation}\label{ee101}
  E_{t,x}\big[\big|\tilde X_{T}-\tilde X_{s}\big|^{q}\big]\le C(1+|x|^{q})|T-s|^{\frac{q}{2}},\qquad
  t\le s<T< T_{0},\ x\in\R^{d}.
\end{equation}
\end{lemma}
\proof We recall the definition of $D_{0}$ in \eqref{e7aa} and notice that
\begin{equation}\label{ee102}
 \left|z\right|=\left|D_{0}\left(\l\right)D_{0}\left(\l^{-1}\right)z\right|\le
 C \l\left|D_{0}\left(\l^{-1}\right)z\right|,\qquad z\in\R^{d},\ 0<\l<1.
\end{equation}
Now we have
\begin{align}
 E_{t,x}\big[\big|\tilde X_{T}-\tilde X_{s}\big|^{q}\big]&=E_{t,x}\left[E_{s,\tilde X_{s}}\big[\big|\tilde X_{T}-\tilde X_{s}\big|^{q}\big]
 \right]
 =\int_{\R^{d}}\tilde\G(t,x;s,\x)\int_{\R^{d}}\tilde\G(s,\x;T,y)\left|y-\x\right|^{q}dy d\x\\
 &\le C\int_{\R^{d}}\bar{\Gamma}_M(t,x;s,\x)\int_{\R^{d}}\bar{\Gamma}_M(s,\x;T,y)\left(\left|y-e^{(T-s)B}\x\right|^{q}+
 \left|\left(e^{(T-s)B}-I\right)\x\right|^{q}\right)dy d\x
\intertext{(applying estimate \eqref{ee102} with $z=y-e^{(T-s)B}\x$, $\l=(T-s)^{\frac{q}{2}}$ and
by Proposition 3.5 in \cite{DiFrancescoPascucci2}, for some $M'>M$)}
 &\le C (T-s)^{\frac{q}{2}} \int_{\R^{d}}\bar{\Gamma}_M(t,x;s,\x)\int_{\R^{d}}\bar{\Gamma}_{M'}(s,\x;T,y)dy
 d\x\\
 &+ C(T-s)^{q}\int_{\R^{d}}\bar{\Gamma}_M(t,x;s,\x)\int_{\R^{d}}|\x|^{q}\bar{\Gamma}_{M}(s,\x;T,y)dy d\x
\end{align}
that yields \eqref{ee101}.\endproof
We finally observe that Theorem \ref{lem:sol_fond_tilde} also implies that $\tilde X$ is a \emph{Feller process on $D$} in the sense of Assumption \ref{assum:feller}, and as such it is a strong Markov process (see \cite{FriedmanSDE1}, Corollary 2.6, p. 28).

\begin{footnotesize}
\bibliographystyle{acm}
\bibliography{Bibtex-Final}
\end{footnotesize}

\end{document}